\documentclass[11pt]{amsart}
\usepackage{amsmath}
\usepackage{amssymb}
\usepackage{amsthm}
\usepackage{amsfonts}
\usepackage{tikz}
\usetikzlibrary{decorations.markings}
\usepackage{graphicx}
\usepackage{hyperref}
\usepackage{fullpage}
\usepackage{float}
\usepackage[hang]{caption}
\usepackage{xcolor}
\usepackage[all]{xy}
\usetikzlibrary{arrows.meta}
\usepackage{mathdots}

\newtheorem{thm}{Theorem}[section]
\newtheorem{prop}[thm]{Proposition}
\newtheorem{cor}[thm]{Corollary}
\newtheorem{lem}[thm]{Lemma}

\theoremstyle{definition}
\newtheorem{defn}[thm]{Definition}
\newtheorem{rmk}[thm]{Remark}
\newtheorem{exmp}[thm]{Example}

\newcommand{\bC}{\mathbb{C}}

\newcommand{\GL}{\mathrm{GL}}
\newcommand{\SL}{\mathrm{SL}}

\newcommand{\cA}{\mathcal{A}}
\newcommand{\cX}{\mathcal{X}}

\newcommand{\cI}{\mathcal{I}}

\newcommand{\DT}{\mathrm{DT}}

\newcommand{\id}{\mathrm{id}}
\newcommand{\bX}{\mathbf{X}}
\newcommand{\bQ}{\mathbb{Q}}
\newcommand{\bk}{\mathbf{k}}
\newcommand{\bA}{\mathbf{A}}
\newcommand{\bW}{\mathbb{W}}
\newcommand{\sgn}{\mathrm{sign}}
\newcommand{\fD}{\mathfrak{D}}
\newcommand{\op}{\mathrm{op}}
\newcommand{\spec}{\mathrm{Spec}}

\colorlet{lightblue}{blue!30!white}
\colorlet{lightred}{red!30!white}
\colorlet{lightteal}{teal!30!white}

\title{F-Polynomials of Donaldson-Thomas Transformations}

\author{Daping Weng}

\begin{document}

\maketitle

\begin{abstract} $F$-polynomials are integer coefficient polynomials encoding the mutations of cluster variables inside a cluster algebra. In this article, we study the $F$-polynomials associated with the action of Donaldson-Thomas transformations on cluster variables. For acyclic quivers, quivers of surface types, and quivers associated with triples of flags, we give explicit descriptions of their Donaldson-Thomas $F$-polynomials in terms of generating functions for ideals inside a labeled poset. We also describe the combinatorial procedure needed to modify these labeled posets to obtain Donaldson-Thomas $F$-polynomials for full subquivers and triangular extensions.
\end{abstract}

\setcounter{tocdepth}{1}
\tableofcontents

\section{Introduction}

Cluster algebra, introduced by Fomin and Zelevinsky \cite{FZI}, is a special family of commutative algebras whose generators and relations are given by a certain recursive procedure called cluster mutations. The combinatorics of a cluster algebra is captured by a family of quivers, which are directed graphs without oriented $1$-cycles and $2$-cycles, related by quiver mutations. Based on the family of mutation equivalent quivers, cluster algebras can be classified into smaller families, such as \emph{finite type} (admitting a Dynkin type quiver), \emph{acyclic type} (admitting a quiver with no oriented cycles), \emph{surface type} (admitting a quiver that comes from a triangulation of a decorated surface), \emph{finite mutation type} (admitting finitely many mutation equivalent quivers up to isomorphisms), and \emph{locally acyclic type} (admitting localizations to acyclic cluster algebras).

Kontsevich and Soibelman \cite{KS} associate a 3d Calabi-Yau category with stability conditions with a quiver (with potential) and construct a Donaldson-Thomas invariant of such categories. This Donaldson-Thomas invariant gives rise to a unique formal automorphism, called the \emph{Donaldson-Thomas transformation}, on the corresponding cluster algebra. Keller \cite{KelDT} characterizes the DT transformation combinatorially by a reddening mutation sequence on the quiver. Using combinatorial methods, reddening mutation sequences have been constructed for finite type and acyclic quviers \cite{BDP}, plabic graph quivers \cite{FordSerhiyenko}, a subfamily of finite mutation type quivers \cite{Mills}, and a subfamily of locally acyclic quivers called Banff quivers \cite{BucherMachacek}. On the other hand, based on Keller's characterization, the DT transformations of many cluster algebras of geometric origins have been constructed geometrically on moduli spaces of $G$-local systems \cite{GS2,GS3}, Grassmannians \cite{Weng}, double Bott-Samelson cells \cite{Wengdb, SWflag}, and braid varieties \cite{CW, CGGLSS}. 

The special generators of a cluster algebra produced by cluster mutations are called \emph{cluster variables}, and they are grouped into overlapping subsets of the same size called \emph{clusters}. One important property of cluster algebras is the \emph{Laurent phenomenon} \cite{FZI}, which states that given a fixed cluster (also known as an \emph{initial cluster}), one can write any other cluster variable as a Laurent polynomial in terms of the initial cluster variables. Moreover, according to the separation formula of Fomin and Zelevinsky \cite{FZIV}, this Laurent polynomial can be further arranged as a product of a Laurent monomial, which is encoded by an integer tuple called the \emph{$G$-vector}, and another Laurent polynomial that can be obtained from a polynomial $F$ via substitution of variables (Theorem \ref{thm: separation formula}). The polynomial $F$ is called an \emph{$F$-polynomial}. Lee and Schiffler \cite{LS} prove that $F$-polynomials have positive integer coefficients. Gross, Hacking, Keel, and Kontsevich \cite{GHKK} give a geometric interpretation of $F$-polynomials as counting certain tropical curves called \emph{broken lines}, which implies not only the positivity result on coefficients but also the fact that $F$-polynomials always have a constant term $1$. However, due to the level of complexity in constructing broken lines, $F$-polynomials are generally poorly understood.

The DT transformation acts on a cluster algebra by permuting its clusters and mapping cluster variables to cluster variables. Thus, one can apply the separation formula to $\DT(A_i)$ for any initial cluster variable $A_i$ and extract a $G$-vector $g_i$ and an $F$-polynomial $F_i$ (see Equation \eqref{eq: DT F polynomial}). It follows from tropical cluster duality \cite{NZ} that the $G$-vector $g_i$ is always a negative basis vector; thus, all the information of the action of $\DT$ is contained in the $F$-polynomial $F_i$. In this article, we focus on some special families of quivers and cluster variables and relate their $F$-polynomials of DT transformations to certain generating functions, which we call \emph{ideal functions}, of ideals in a labeled poset (Definitions \ref{defn: ideals} \ref{defn: labeled posets}, \ref{defn: ideal functions}). Below is a summary of our main results.

\begin{thm}\label{thm: main theorem} For the following families of quivers and vertices, their $F$-polynomials of DT transformations are given by ideal functions of simply-labeled pointed posets:
\begin{center}
\begin{tabular}{|c|c|c|c|} \hline
    \emph{Quiver Types} & \emph{Vertices} & \emph{Posets} & \emph{Theorems} \\ \hline
    \emph{acyclic} & \emph{any vertex} & \emph{ascendant tree} & \emph{\ref{thm: DT for acyclic quivers}} \\ \hline
    \emph{surface type} & \emph{admissible arc} & \emph{Figure \ref{fig: Hasse diagrams for surface quivers} } & \emph{\ref{thm: DT for surfaces}} \\ \hline
    \emph{triple of flags} & \emph{any vertex} & \emph{3d Young diagram} & \emph{\ref{thm: DT for triple of flags}} \\ \hline
\end{tabular}
   \end{center}
Moreover, for quivers that are constructed by taking full subquivers or triangular extensions, their $F$-polynomials can be constructed by a corresponding truncation \emph{(Corollary \ref{cor: truncation})} or extension \emph{(Corollary \ref{cor: extension})}.
\end{thm}

Because Theorem \ref{thm: main theorem} relates the $F$-polynomials of DT transformations to ideal functions of simply-labeled pointed posets, the $F$-polynomial facts about positive coefficients and constant term $1$ become apparent.

\begin{exmp} As a demonstration, consider the following quiver on the left in Figure \ref{fig: first example}. Note that the full subquiver spanned by the vertices $\{0,1,\dots, 6\}$ is a full subquiver of the quiver associated with a triple of flags in $\bC^6$, the full subquiver spanned by the vertices $\{7,8,9\}$ is a quiver of surface type, and the whole quiver is a triangular extension of these two full subquivers. Thus, by applying Theorem \ref{thm: main theorem}, we can write down a labeled poset, whose ideal function is the $F$-polynomial of DT, for every vertex in this quiver. For example, such a labeled poset for the vertex $0$ is drawn on the right in Figure \ref{fig: first example}. Note that it is obtained by attaching three linear chains to a cube-shaped 3d Young diagram poset.
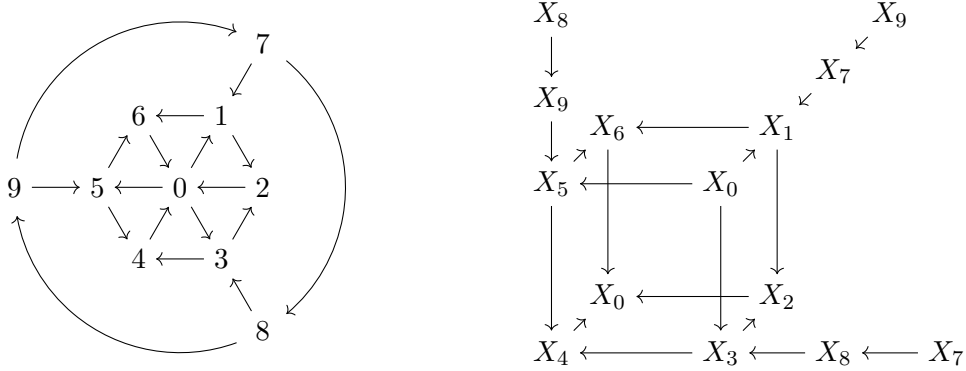
\begin{figure}[H]
    \centering
    \begin{tikzpicture}[baseline=0,scale=1.1]
    \foreach \i in {1,...,6}
    {
    \node (\i) at (120-\i*60:1) [] {$\i$};
    }
    \node (0) at (0,0) [] {$0$};
    \foreach \i in {7,8,9}
    {
    \node (\i) at (180-\i*120:2) [] {$\i$};
    }
    \draw [->] (2) -- (0);
    \draw [->] (0) -- (5);
    \draw [->] (6) -- (0);
    \draw [->] (0) -- (3);
    \draw [->] (4) -- (0);
    \draw [->] (0) -- (1);
    \draw [->] (1) -- (6);
    \draw [->] (5) -- (6);
    \draw [->] (1) -- (2);
    \draw [->] (3) -- (2);
    \draw [->] (3) -- (4);
    \draw [->] (5) -- (4);
    \draw [->] (7) -- (1);
    \draw [->] (8) -- (3);
    \draw [->] (9) -- (5);
    \draw [->] (50:2) arc (50:-50:2);
    \draw [->] (-70:2) arc (-70:-170:2);
    \draw [->] (170:2) arc (170:70:2);
    \end{tikzpicture} \hspace{2cm}
    \begin{tikzpicture}[baseline=20, scale=0.75]
        \node (0) at (-1,-1) [] {$X_0$};
        \node (4) at (-2,-2) [] {$X_4$};
        \node (3) at (1,-2) [] {$X_3$};
        \node (2) at (2,-1) [] {$X_2$};
        \node (5) at (-2,1) [] {$X_5$};
        \node (6) at (-1,2) [] {$X_6$};
        \node (1) at (2,2) [] {$X_1$};
        \node (00) at (1,1) [] {$X_0$};
        \node (17) at (3,3) [] {$X_7$};
        \node (19) at (4,4) [] {$X_9$};
        \node (38) at (3,-2) [] {$X_8$};
        \node (37) at (5,-2) [] {$X_7$};
        \node (59) at (-2,2.5) [] {$X_9$};
        \node (58) at (-2,4) [] {$X_8$};
        \draw [->] (2) -- (0);
    \draw [->] (00) -- (5);
    \draw [->] (6) -- (0);
    \draw [->] (00) -- (3);
    \draw [->] (4) -- (0);
    \draw [->] (00) -- (1);
    \draw [->] (1) -- (6);
    \draw [->] (5) -- (6);
    \draw [->] (1) -- (2);
    \draw [->] (3) -- (2);
    \draw [->] (3) -- (4);
    \draw [->] (5) -- (4);
    \draw [->] (19) -- (17);
    \draw [->] (17) -- (1);
    \draw [->] (37) -- (38);
    \draw [->] (38) -- (3);
    \draw [->] (58) -- (59);
    \draw [->] (59) -- (5);
    \end{tikzpicture}
    \caption{Example of a quiver and the labeled poset that describes the $F$-polynomial of DT for one of the quiver vertices.}
    \label{fig: first example}
\end{figure}
\end{exmp}

\begin{exmp} Based on an idea of Shen, Zhou \cite{Zhou} studies how DT transformations are related under quiver folding. By combining our result with Zhou's result, we can even write down the $F$-polynomials of the DT transformation on the Markov quiver (left picture in Figure \ref{fig: markov}), which is known to not have a reddening sequence. The trick is to observe that there is a 2-to-1 covering from the quiver associated with a 4-punctured sphere (middle picture in Figure \ref{fig: markov}) to the Markov quiver. Note that the covering quiver admits a reddening sequence and Theorem \ref{thm: main theorem} is applicable to all vertices of this covering quiver. Then to obtain the $F$-polynomials of DT for the Markov quiver, we just need to substitute the labelings of the poset by their counterpart in the folded quiver (Markov quiver). We draw the Hasse diagram for one such labeled poset on the right in Figure \ref{fig: markov}. The $F$-polynomial of DT for the other two vertices can be obtained by cyclic symmetry.
\begin{figure}[H]
    \centering
    \begin{tikzpicture}
        \foreach \i in {1,2,3}
        {
        \node at (90-\i*120:1.5) [] {$\i$};
        \draw [->] (80-\i*120:1.4) -- (-20-\i*120:1.4);
        \draw [->] (85-\i*120:1.1) -- (-25-\i*120:1.1);
        }
    \end{tikzpicture}\hspace{2cm}
    \begin{tikzpicture}[scale=1.3]
        \node (11) at (0,1) [] {$1$};
        \node (12) at (0,-1) [] {$1'$};
        \node (21) at (-1.25,-0.25) [] {$2$};
        \node (22) at (1.25,0.25) [] {$2'$};
        \node (31) at (-0.5,0.25) [] {$3$};
        \node (32) at (0.5,-0.25) [] {$3'$};
        \foreach \i in {1,2}
        {
        \foreach \j in {1,2}
        {
        \draw[->] (1\i) -- (2\j);
        \draw [->] (2\i) -- (3\j);
        \draw [->] (3\i) -- (1\j);
        }
        }
    \end{tikzpicture}\hspace{2cm}
    \begin{tikzpicture}
        \node (0) at (0,-0.6) [] {$X_1$};
        \node (31) at (-1,0) [] {$X_3$};
        \node (32) at (1,0) [] {$X_3$};
        \node (21) at (-1,1) [] {$X_2$};
        \node (22) at (1,1) [] {$X_2$};
        \node (1)  at (0,1.6) [] {$X_1$};
        \foreach \i in {1,2}
        {
        \draw [->] (3\i) -- (0);
        \draw [->] (1) -- (2\i);
        \foreach \j in {1,2}
        {
        \draw [->] (2\i) -- (3\j);
        }
        }
    \end{tikzpicture}
    \caption{Labeled posets for $F$-polynomials of DT on the Markov quiver.}
    \label{fig: markov}
\end{figure}
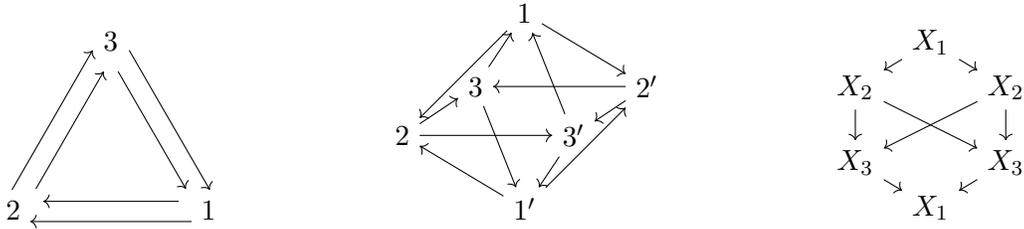
\end{exmp}

\subsection*{Acknowledgement} We would like to thank Alexander Goncharov for his guidance and encouragement in our study of the DT transformation. We would also like to thank Linhui Shen for the discussion on the relations between DT transformations of subquivers/folded quivers and the original quiver.

\section{Ideal Functions from Labeled Posets}

A \emph{poset} (short for ``partially ordered set'') is a finite set $P$ with a partial order $\leq$. A poset is said to be \emph{pointed} if there is a unique minimal element in $P$. 

\begin{defn}\label{defn: ideals} An \emph{ideal} of a poset $P$ is a subset $I\subset P$ satisfying
\[
(i\in I \text{ and } j\leq i)\implies (j\in I).
\]
Note that we allow ideals to be empty. Given a collection of elements in $P$, the ideal \emph{generated} by these elements is the intersection of all ideals containing these elements. We denote the set of ideals in a poset $P$ by $\cI(P)$.
\end{defn}

Let $\bX$ be a set of formal commuting variables, and let $\mathbb{Q}(\bX)$ be the field of rational functions with rational coefficients in the variables in $\bX$. 

\begin{defn}\label{defn: labeled posets} An \emph{$\bQ(\bX)$-labeling}, or just \emph{labeling}, of a poset $P$ is a map $L:P\rightarrow \bQ(\bX)$. We do not assume this map to be injective. We call the pair $(P,L)$ a \emph{labeled poset}. If the image of the labeling $L$ lies inside the subset $\bX\subset \bQ(\bX)$, then we say $L$ is \emph{simple}. A simple poset labeled by a simple labeling is called a \emph{simply-labeled poset}. Note that when the labeling $L$ is simple, the ideal function $F_{(P,L)}$ naturally lives inside $\mathbb{Z}[\bX]$ with positive coefficients. To simplify the notation, we define $L_i:=L(i)$ for any labeling $L$.
\end{defn}

\begin{defn}\label{defn: ideal functions} Let $(P,L)$ be a labeled poset. The \emph{ideal function} associated with this labeled poset is defined to be the function
\[
F_{(P,L)}:=\sum_{I\subset \cI(P)}L_I,  \quad \text{where} \quad L_I:=\prod_{i\in I} L_i.
\]
We view $F_{(P,L)}$ as an element in $\bQ(\bX)$. In particular, the empty ideal always contributes a constant term $1$ to the ideal function.
\end{defn}

To better visualize ideal functions, we can make use of Hasse diagrams. The \emph{Hasse diagram} of a poset $P$ is a quiver whose vertex set is $P$, with an arrow from $i$ to $j$ whenever $i>j$ and there is no $k$ such that $i>k>j$. Note that Hasse diagrams are automatically acyclic quivers. Moreover, we can draw a Hasse diagram so that the arrows are always oriented downward. Now given an additional $\bQ(\bX)$-labeling $L:P\rightarrow \bQ(\bX)$, we can replace each vertex $i$ in the Hasse diagram with its image $L_i$. To write the ideal function, we just need to read from bottom up along the Hasse diagram to find all the ideals in $P$, and each ideal is going to contribute a term in the ideal function. We adopt the convention that if we represent a labeled poset with a Hasse diagram, then its ideal function is denoted by adding a square bracket around the Hasse diagram.

\begin{exmp} Consider the following simply-labeled poset. Its ideal function is
\[
\left[\begin{tikzpicture}[baseline=40]
    \node (1) at (0,0) [] {$X_1$};
    \node (2) at (1,1) [] {$X_2$};
    \node (3) at (-1,1) [] {$X_3$};
    \node (4) at (0,2) [] {$X_4$};
    \node (5) at (0,3) []{$X_1$};
    \draw [->] (5) -- (4);
    \draw [->] (4) -- (3);
    \draw [->] (4) -- (2);
    \draw [->] (2) -- (1);
    \draw [->] (3) -- (1);
    \end{tikzpicture}\right]=1+X_1+X_1X_2+X_1X_3+X_1X_2X_3+X_1X_2X_3X_4+X_1^2X_2X_3X_4.
\]
\end{exmp}

Let us prove some basic lemmas about ideal functions.

\begin{lem}\label{lem: formula for pointed labeled posets} Let $(P,L)$ be a labeled poset and suppose $P$ is pointed. Let $P'$ be obtained from $P$ by removing the unique minimal element $i_{\min}$. Let $(P',L')$ be the labeled poset obtained by restricting the partial order and the labeling to $P'$. Then
\[
F_{(P,L)} = 1+L_{i_{\min}}F_{(P',L')}.
\]
\end{lem}
\begin{proof} It follows from the fact that any non-empty ideal in $P$ contains the minimal element $i_{\min}$.
\end{proof}

\begin{lem}\label{lem: insertion} 
Let $(P,L)$ be a labeled poset. Let $j<i$ be a pair of vertices in $P$. Let $P':=P\sqcup\{k\}$ and we extend the partial order to $P'$ by declaring $j<k<i$. We choose an invertible element $L_0\in \bQ(\bX)$ such that $1+L_0$ is also invertible, and define a new labeling $L'$ on $P'$ by 
\[
L'_s:=\left\{\begin{array}{ll}
    L_j(1+L_0)^{-1} &  \text{if $s=j$},\\
    L_0 & \text{if $s=k$}, \\
    L_i(1+L_0^{-1}) & \text{if $s=i$},\\
    L_s & \text{otherwise}.
\end{array}\right.
\]
Then
\[
F_{(P',L')}=F_{(P,L)}.
\]
\end{lem}
\begin{proof} The ideals in $P$ belong to three mutually exclusive families.

For those ideals $I$ in $P$ that do not contain $j$, they are also ideals in $P'$. It follows from the construction of $P'$ that $L_I=L'_I$ for these ideals.

For those ideals $I$ in $P$ that contains $j$ but not $i$, each of them corresponds to two ideals in $P'$, namely $I$ and $I\sqcup\{k\}$. By construction, we have
\[
L'_I+L'_{I\sqcup\{k\}}=L_I(1+L_0)^{-1}+L_I(1+L_0)^{-1}L_0=L_I.
\]

For any ideal $I$ in $P$ that contains both $i$ and $j$, $I\sqcup\{k\}$ is an ideal in $P'$, and we have
\[
L'_{I\sqcup\{k\}}=L_I(1+L_0)^{-1}L_0(1+L_0^{-1})=L_I.
\]

Note that these three cases exhaust all ideals in $P'$ as well. This implies $F_{(P',L')}=F_{(P,L)}$.
\end{proof}

Note that Lemma \ref{lem: insertion} can be used in the reverse direction as well. If $i\rightarrow k\rightarrow j$ are arrows in a Hasse diagram $H$ of a labeled poset, with $k$ not connecting to any other vertex, then we can remove $k$ from the poset $P$ and relabel the vertices $i$ and $j$ without changing the ideal function.

\bigskip

If the labeling map $L$ labels some poset elements as $0$, then we can further simplify the labeled poset without changing the ideal function.

\begin{defn} Let $P$ be a poset. The \emph{opposite poset} $P^\op$ is the poset with the same underlying set as $P$ but with the opposite order. An ideal in $P^\op$ is called an \emph{anti-ideal} in $P$. Given a collection of elements in $P$, the anti-ideal \emph{generated} by these elements is the same as the ideal generated by the same collection of elements in $P^\op$.
\end{defn}

\begin{lem}\label{lem: setting to 0} Suppose $(P,L)$ is a labeled poset and suppose $S\subset P$ such that $L|_S$ is the constant zero map. Let $P_S$ be the complement to the anti-ideal generated by $S$ and let $L_S:=L|_{P_S}$. Then 
\[
F_{(P,L)}=F_{(P_S,L_S)}.
\]
\end{lem}
\begin{proof} Let $A_S$ be the anti-ideal generated by $S$. If an ideal $I$ in $P$ intersects $A_S$ non-trivially, then $I$ must contain some element $i\in S$. Since $L|_S=0$, this implies that $L_I=0$. Thus, the only non-zero terms in $F_{(P,L)}$ correspond to ideals $I$ that are disjoint from $A_S$. But they are the same as ideals in $P_S$. Therefore we can conclude that $F_{(P,L)}=F_{(P_S,L_S)}$.
\end{proof}

\section{Overview of Quivers and Cluster Algebras}

Within this article, we always assume quivers do not have oriented $1$-cycles or oriented $2$-cycles. We denote the set of vertices of $Q$ by $I$ and record the arrows in $Q$ by its \emph{exchange matrix} $\epsilon$, an $I\times I$ skew-symmetric matrix whose entries are
\[
\epsilon_{ij}:=\#(i\rightarrow j)-\#(j\rightarrow i).
\]

Given a vertex $k$ of $Q$, we can \emph{mutate} $Q$ in the direction of $k$ to get a new quiver $Q'$ with the same set of vertices, but with its exchange matrix changed to
\[
\epsilon'_{ij}=\left\{\begin{array}{ll} -\epsilon_{ij} & \text{if $k\in \{i,j\}$},\\
\epsilon_{ij}+[\epsilon_{ik}]_+[\epsilon_{kj}]_+-[-\epsilon_{ik}]_+[-\epsilon_{kj}]_+ & \text{otherwise.}
\end{array}\right.
\]
Here $[a]_+:=\max\{a,0\}$. Quivers obtained from $Q$ via a sequence of mutations are said to be \emph{mutation equivalent} to $Q$. Mutation equivalence is an equivalence relation, and we denote it by $\sim$.

The quiver with \emph{principal coefficient} associated with $Q$ is a quiver $\tilde{Q}$ constructed from $Q$ by adding a single frozen vertex $i'$ for each vertex $i$ in $Q$, together with a single arrow pointing from $i$ to $i'$. Note that we are not allowed to mutate at frozen vertices. The exchange matrix of $\tilde{Q}$ is hence of the form $\begin{pmatrix} \epsilon & \id \\ -\id & 0\end{pmatrix}$. Now for a given sequence of quiver vertices $\mathbf{k}=(k_1,k_2,\dots, k_l)$ of $Q$, we can mutate the quiver $\tilde{Q}$ along this sequence and obtain another quiver $\tilde{Q}'=\mu_\bk(\tilde{Q})=\mu_{k_l}\circ\mu_{k_{l-1}}\circ \cdots \circ \mu_{k_1}(\tilde{Q})$. The exchange matrix of $\tilde{Q}'$ is of the form $\begin{pmatrix} \epsilon' & c \\ -c & \ast\end{pmatrix}$, where $\epsilon'$ is the exchange matrix of the quiver $\mu_\bk(Q)$. The $I\times I$ matrix $c$ is called the \emph{$C$-matrix} associated with $Q$ and the mutation sequence $\bk$. The row vectors of the $C$-matrix are called \emph{$C$-vectors}. In particular, the matrix $c$ satisfies an exceptional sign coherence property; this property was conjectured by Fomin and Zelevinsky in \cite{FZIV}, proved in the skew-symmetric case by Derksen, Weyman, and Zelevinsky and in the full generality by Gross, Hacking, Keel, and Kontsevich.

\begin{thm}[Sign coherence \cite{DWZ,GHKK}]\label{thm: sign coherence} For each $C$-vector $c_i$, either all entries in that row vector are non-negative or all entries in that row vector are non-positive.
\end{thm} 

Using this exceptional property, Keller \cite{KelDT} defines a coloring on vertices of any quiver mutation equivalent to an initial quiver $Q$; based on this coloring scheme, he further defines special mutation sequences called ``reddening sequences'' and ``maximal green sequences''.

\begin{defn} Let $Q$ be a quiver and let $\bk$ be a mutation sequence. Let $c$ be the $C$-matrix associated with the quiver $Q$ and the mutation sequence $\bk$. A vertex $i$ in $\mu_\bk(Q)$ is \emph{green} if all entries in $c_i$ are non-negative, and is \emph{red} if all entries in $c_i$ are non-positive. Note that by Theorem \ref{thm: sign coherence} any vertex in $\mu_\bk(Q)$ is either green or red.
\end{defn}

\begin{defn} A mutation sequence $\mu_\bk$ on an initial quiver $Q$ is said to be \emph{reddening} if all vertices in $\mu_\bk(Q)$ are red. A reddening sequence is said to be \emph{maximal green} if it only mutates at green vertices at each step.
\end{defn}

To each family of mutation equivalent quivers, Fomin and Zelevinsky \cite{FZI} introduce a commutative algebra called the ``cluster algebra''. To construct the cluster algebra with an initial quiver $Q$, we first associate a formal variable $A_i$ with each vertex $i$ of $Q$. Let $\bQ(\bA)$ be the field of rational functions in the set of formal variables $\bA=\{A_i\}_{i\in I}$. Now for any mutation sequence $\mu_\bk$, we will follow a recursive procedure to produce new elements in $\bQ(\bA)$. Suppose the elements we obtain for $Q'=\mu_{k_{l-1}}\circ \cdots \circ \mu_{k_1}(Q)$ are $\{A'_i\}_{i\in I}$, and the exchange matrix of $Q'$ is $\epsilon'$; then for $Q''=\mu_{k_l}(Q')$, we produce a new collection by the following \emph{mutation formula}:
\begin{equation}\label{eq: A mutation}
A''_i=\left\{\begin{array}{ll}
\displaystyle\frac{1}{A'_k}\left(\prod_j A'^{[\epsilon'_{kj}]_+}_j+\prod_jA'^{[-\epsilon'_{kj}]_+}_j\right) & \text{if $i=k$},\\
A'_i & \text{otherwise}.\end{array}\right.
\end{equation}
Such a collection of elements is called a \emph{cluster} and elements inside a cluster are called \emph{cluster variables}. 

\begin{defn}[{\cite[Definition 2.3]{FZI}}] The cluster algebra $\cA$ associated with (the family of quivers mutation equivalent to) the quiver $Q$ is the subalgebra inside $\bQ(\bA)$ generated by all cluster variables.
\end{defn}

Fock and Goncharov \cite{FGensemble} introduce a dual version of cluster algebras called the \emph{cluster Poisson algebra}. The construction of a cluster Poisson algebra is only slightly different from that of cluster algebra. We again start with a collection of formal variables $\bX=\{X_i\}_{i\in I}$, and we mutate them along mutation sequences to produce a new collection of elements in $\bQ(\bX)$ for each quiver in the mutation equivalent family. In contrast to the mutation of cluster variables \eqref{eq: A mutation}, the cluster Poisson variables mutate according to {\cite[Equation (13)]{FGensemble},\cite[Proposition 3.9]{FZIV}}:
\begin{equation}\label{eq: X mutation}
X''_i=\left\{\begin{array}{ll} X'^{-1}_k & \text{if $i=k$},\\
X'_i\left(1+X'^{-\sgn(\epsilon'_{ik})}_k\right)^{-\epsilon'_{ik}} & \text{otherwise.} \end{array}\right.
\end{equation}
For each quiver $Q'\sim Q$, we can take the Laurent polynomial ring $\bQ[X'^{\pm 1}_i]_{i\in I}$. The \emph{cluster Poisson algebra} is then defined to be the intersection
\[
\mathcal{X}:=\bigcap_{Q' \sim Q}\bQ[X'^{\pm 1}_i]_{i\in I}.
\]
Note that, unlike cluster algebras, cluster Poisson variables may not be elements inside the cluster Poisson algebra.

For any quiver $Q'\sim Q$ with an exchange matrix $\epsilon'$, we have a map $p:\bQ(\bX)\rightarrow \bQ(\bA)$ defined by 
\begin{equation} \label{eq: p map}
p(X'_i)=\prod_{j\in I}A'^{\epsilon'_{ij}}_j,
\end{equation}
where $A'_j$'s and $X_i$'s are variables associated with the quiver $Q'$. By comparing the mutation formulas \eqref{eq: A mutation} and \eqref{eq: X mutation}, it is not hard to see that the $p$ maps commute with mutations and hence there is a unique $p$ map $p:\bQ(\bX)\rightarrow \bQ(\bA)$ regardless of which quiver we use.

Let $Q'=\mu_\bk(Q)$ be a quiver mutation equivalent to $Q$. Let $A'_i$ be cluster variables and let $X'_i$ be cluster Poisson variables associated with $Q'$. Then $A'_i$ and $X'_i$ can be written in terms of the initial cluster variables and the initial cluster Poisson variables (those associated with $Q$) as follows.

\begin{thm}[Separation Formula {\cite[Proposition 3.13, Corollary 6.3]{FZIV}}]\label{thm: separation formula} There exists a collection of integers $g_{ij}$ and polynomials $F_{A'_i}(\bX)$ in the initial cluster Poisson variables $\bX$ with integer coefficients such that
\[
A'_i=\left(\prod_j A_j^{g_{ij}}\right)p(F_{A'_i}(\bX)) \quad \quad \text{and} \quad \quad X'_i=\prod_j\left(X_j^{c_{ij}}F_{A'_j}(\bX)^{\epsilon'_{ij}}\right).
\]
The integers $c_{ij}$ are entries of the corresponding $C$-matrix.
\end{thm}

\begin{defn} The integers $g_{ij}$ form a matrix $g$ called the \emph{$G$-matrix} associated with the mutation sequence $\mu_\bk$ on the initial quiver $Q$. The row vectors of the $G$-matrix are called \emph{$G$-vectors}. The polynomial $F_{A'_i}$ is called the \emph{$F$-polynomial} associated with the cluster variable $A'_i$ with respect to the initial quiver $Q$.
\end{defn}

$F$-polynomials have many interesting properties. Below are a couple of them.

\begin{thm}[Positivity {\cite[Theorem 1.1]{LS}}] $F$-polynomials have positive integer coefficients.
\end{thm}

\begin{thm}[Constant Term {\cite[Corollary 5.5]{GHKK}}] The constant term of an $F$-polynomial is $1$.
\end{thm}

On the other hand, the $G$-matrix is closely related to its corresponding $C$-matrix via the following theorem.

\begin{thm}[Tropical cluster duality {\cite[Theorem 1.2]{NZ}}]\label{thm: tropical cluster duality} $g^t=c^{-1}$. As a corollary, the $G$-matrix $g$ also enjoys a sign coherence along its column vectors.
\end{thm}

Recall that two quivers $Q$ and $Q'$ are said to be \emph{isomorphic} if there is a bijection $\sigma:I\rightarrow I'$ such that $\epsilon'_{\sigma(i)\sigma(j)}=\epsilon_{ij}$. If $\mu_\bk$ is a mutation sequence on an initial quiver $Q$ such that $Q'=\mu_\bk(Q)$ is isomorphic to $Q$ via $\sigma$, then $\sigma$ induces an automorphism on the cluster algebra $\cA$ and an automorphism on the cluster Poisson algebra $\cX$. By abuse of notation, we also denote these automorphisms by $\sigma$, and it is defined on the cluster variables and cluster Poisson variables of $Q'$ and $Q$ by
\[
\sigma(A_i)=A'_{\sigma(i)} \quad \quad \text{and} \quad \quad \sigma(X_i)=X'_{\sigma(i)}.
\]
Automorphisms defined by isomorphisms between mutation equivalent quivers are called \emph{cluster transformations}.

\begin{thm}[{\cite[Theorem 3.4]{GS2}}] If $\mu_\bk$ is a reddening sequence on an initial quiver $Q$, then $\mu_\bk(Q)$ is isomorphic to $Q$. In particular, there is a unique isomorphism $\sigma$ such that $c_{i\sigma(j)}=-\delta_{ij}$.
\end{thm}

\begin{defn} The cluster transformation defined by a reddening sequence on $Q$ (together with the unique isomorphism $\sigma$) is called a \emph{cluster Donaldson-Thomas transformation}; we denote cluster Donaldson-Thomas transformation by $\DT$.
\end{defn}

\begin{thm}[{\cite[Corollary 3.7]{GS2}}] A cluster DT transformation is unique if it exists. In particular, $\DT$ commutes with cluster mutations and does not depend on the choice of an initial quiver.
\end{thm} 

Since cluster DT transformations map cluster variables to cluster variables, we can apply the separation formula (Theorem \ref{thm: separation formula}) to the cluster variable $\DT(A_i)$ and the cluster Poisson variable $\DT(X_i)$. Moreover, because of the tropical cluster duality (Theorem \ref{thm: tropical cluster duality}), we know that the integers $g_{ij}=-\delta_{ij}$; therefore we have
\begin{equation}\label{eq: DT F polynomial}
\DT(A_i)=A_i^{-1}p(F_i(\bX)) \quad \quad \text{and} \quad \quad \DT(X_i)=X_i^{-1}\prod_jF_j(\bX)^{\epsilon_{ij}}.
\end{equation}
Note that since $\DT$ is unique, the $F$-polynomials $F_i$ do not depend on the choice of reddening sequence. In conclusion, for any quiver $Q$ with cluster DT transformations (i.e., reddening sequences), there is a unique collection of $F$-polynomials for the cluster DT transformation, and conversely, the information of the cluster DT transformation is completely determined by its $F$-polynomials with respect to an initial quiver $Q$.

\section{Properties of \texorpdfstring{$F$-Polynomials for DT Transformations}{}}

Similar to cluster variables, the $F$-polynomials themselves also have a mutation formula. The $F$-polynomial associated with every initial cluster variable is the constant polynomial $1$. Now suppose we have obtained the collection of $F$ polynomials $\{F'_i\}$ for a certain quiver $Q'$ with exchange matrix $\epsilon'$ and $C$-matrix $c'$. Then mutating again in the direction $k$ keeps all $F$-polynomial $F'_i$ with $i\neq k$ unchanged, and changes $F'_k$ to {\cite[Proposition 5.1]{FZIV}}:
\begin{equation}\label{eq: F mutation}
F''_k=\mu_k(F'_k)=\frac{1}{F'_k}\left(\prod_jX_j^{[c'_{kj}]_+}F'^{[\epsilon'_{kj}]_+}_j+\prod_jX_j^{[-c'_{kj}]_+}F'^{[-\epsilon'_{kj}]_+}_j\right).
\end{equation}
Note that by sign coherence (Theorem \ref{thm: sign coherence}), one and only one of $\prod_jX_j^{[c'_{kj}]_+}$ and $\prod_jX_j^{[-c'_{kj}]_+}$ is non-trivial.

Now let us focus on the $F$-polynomials for cluster DT transformations \eqref{eq: DT F polynomial}. Suppose $Q'=\mu_k(Q)$. We may use $Q'$ instead of $Q$ as an initial quiver, and there is a different collection of $F$-polynomials $\{F'_i(\bX')\}_{i\in I}$ for $\DT$ with respect to $Q'$ as the initial quiver. Please be aware that the variables $\bX'=\{X_i\}_{i\in I}$ associated with $Q'$ are different from those associated with $Q$. By taking the inverse of the cluster Poisson mutation formula \eqref{eq: X mutation}, we get the following relation:
\begin{equation}\label{eq: reverse X mutation}
X_i=\left\{\begin{array}{ll} X'^{-1}_k & \text{if $i=k$},\\
X'_i\left(1+X'^{\sgn(\epsilon_{ik})}_k\right)^{\epsilon_{ik}} & \text{otherwise.} \end{array}\right.    
\end{equation}
Given an expression $F$ (e.g. polynomial) in the $\bX$ variables, we can use Equation \ref{eq: reverse X mutation} to rewrite it in terms of the $\bX'$ variables. We denote the resulting expression by $\widehat{F}$.

\begin{prop}[Relating $F$-polynomials of DT for adjacent clusters]\label{prop: DT for adjacent clustsers} Suppose $Q'=\mu_k(Q)$. The $F$-polynomails $\{F'_i(\bX')\}_{i\in I}$ of $\DT$ with respect to $Q$' are related to the $F$-polynomials $\{F_i(\bX)\}_{i\in I}$ of $\DT$ with respect to $Q$ by
\[
F'_i(\bX')=\left\{\begin{array}{ll} 
(1+X'_k)\widehat{\mu_k(F_k)} & \text{if $i=k$},\\ \quad & \\
\widehat{F}_i & \text{otherwise}.
\end{array}\right.
\]
\end{prop}
\begin{proof} Suppose we first use $Q$ as the initial quiver and let $\mu_\bk$ be a reddening sequence on $Q$. For simplicity, let us assume that the quiver isomorphism $\sigma$ relating $Q$ and $\mu_\bk(Q)$ is the identity map. Now after the reddening sequence $\mu_\bk$, we can mutate one more time in the direction of $k$. The $F$-polynomials after this extra mutation are all identical to the $F$-polynomials of the $\DT$ transformation except for $F_k$, which mutates according to \eqref{eq: F mutation}. 

Next, we would like to change the initial seed to $Q'$. Note that for $i\neq k$, the initial cluster variable $A'_i$ is identical to the initial cluster variable $A_i$ of $Q$. Therefore
\[
\DT(A'_i)=\DT(A_i)=A_i^{-1}p(F_i)=A'^{-1}_ip(F_i),
\]
and hence the $F$-polynomial $F'_i$ is just $F_i$ after a substitution according to \eqref{eq: reverse X mutation}, i.e., $\widehat{F}_i$.

Now consider the cluster variable $A'_k$. According to \eqref{eq: A mutation}, and \eqref{eq: DT F polynomial}, we have
\begin{align*}
\DT(A'_k)=&\frac{A_k}{p(F_k)}\left(\prod_jA_j^{-[\epsilon_{kj}]_+}p(F_j)^{[\epsilon_{kj}]_+}+\prod_j A_j^{-[-\epsilon_{kj}]_+}p(F_j)^{[-\epsilon_{kj}]_+}\right)\\
=&A_k\left(\prod_j A_j^{-[-\epsilon_{kj}]_+}\right)\frac{1}{p(F_k)}\left(\prod_jA_j^{-\epsilon_{kj}}p(F_j)^{[\epsilon_{kj}]_+}+\prod_jp(F_j)^{[-\epsilon_{kj}]_+}\right)\\
=&A_k\left(\prod_j A_j^{-[-\epsilon_{kj}]_+}\right)\frac{p(X_k)^{-1}}{p(F_k)}\left(\prod_jp(F_j)^{[\epsilon_{kj}]_+}+p(X_k)\prod_jp(F_j)^{[-\epsilon_{kj}]_+}\right)
\end{align*}
Note that since the $C$-matrix after $\DT$ is $-\id$. By comparing with the mutation formula for $F$-polynomials \eqref{eq: F mutation}, the last expression is in turn equal to
\begin{equation}\label{eq1}
\DT(A'_k)=A_k\left(\prod_j A_j^{-[-\epsilon_{kj}]_+}\right)p(X_k)^{-1}p(\mu_k(F_k)).
\end{equation}
On the other hand, from \eqref{eq: A mutation} we can also deduce that
\[
A_k\prod_j A_j^{-[-\epsilon_{kj}]_+}=A'^{-1}_k\left(1+\prod_j A_j^{\epsilon_{kj}}\right)=A'^{-1}_k\left(1+p(X_k)\right).
\]
By plugging this into \eqref{eq1}, we get
\[
\DT(A'_k)=A'^{-1}_k\left(1+p(X_k)^{-1}\right)p(\mu_k(F_k))=A'^{-1}_kp\left((1+X'_k)\widehat{\mu_k(F_k)}\right).
\]
From this we can conclude that the $F$-polynomial of $\DT$ with respect to the initial quiver $Q'$ is $F'_k=(1+X'_k)\widehat{\mu_k(F_k)}$.
\end{proof}

Next, let us consider $F$-polynomials of cluster DT transformations on subquivers. Let $Q$ be a quiver with vertex set $I$ and exchange matrix $\epsilon$. A \emph{full subquiver} $Q'$ of $Q$ is defined by a non-empty subset $I'$ of $I$ with an exchange matrix $\epsilon'=\epsilon|_{I'\times I'}$. Regarding cluster DT transformations on subquivers, we have the following theorem.

\begin{thm}[{\cite[Theorem 1.4.1]{Muller_max_green}}] If $Q$ admits a cluster DT transformation, then so does any full subquiver $Q'$ of $Q$.
\end{thm}

However, it is in general not true that if $\mu_\bk$ is a reddening sequence on $Q$, then the restriction of $\mu_\bk$ to $Q'$ will also be a reddening sequence. Nevertheless, we can still relate the $F$-polynomials of $\DT$ on a full subquiver $Q'$ to those of $\DT$ on the whole quiver $Q$.

\begin{prop}[Relating $F$-polynoimials of $\DT$ for full subquivers]\label{prop: DT for subquivers} Let $Q'$ be a subquiver of $Q$ defined by a subset $I'\subset I$. Suppose the cluster DT transformation exists on $Q$. Then the $F$-polynomials of the DT transformation on $Q'$ can be obtained from those for $Q$ with indices in $I'$ by substituting $X_j=0$ for all $j\notin I'$.
\end{prop}
\begin{proof} The proof of this proposition is based on the idea of degeneration of scattering diagrams \cite{BFMN}. By following the construction in \cite{GHKK}, one can construct a scattering diagram from any quiver $Q$ with a choice of functions attached to the initial walls. To obtain a degeneration of the scattering diagram, we choose to attach the function $1+tX_j$ instead of $1+X_j$ to the $j$th initial wall for each $j\notin I'$. We denote the resulting consistent scattering diagram by $\fD_t$. Note that $\fD_1$ is the same as the ordinary cluster scattering diagram $\fD_Q$ for the quiver $Q$, and $\fD_0$ is the product of the ordinary cluster scattering diagram $\fD_{Q'}$ with a lineality space $\mathbb{R}^{|I\setminus I'|}$.

The $\DT$ transformation corresponds to a path $\gamma$ that goes from the positive chamber to the negative chamber of the scattering diagram $\fD_t$, and for each cluster variable $A_i$, $\DT_t(A_i)$ is equal to the outcome of the wall-crossing transformation of $A_i$ along the path $\gamma$. With the extra parameter $t$, $\DT_t(A_i)$ is now a polynomial in $t$ (it is a polynomial because $t$ only appears along with the variables $X_j$ for $j\notin I'$ inside the $F$-polynomial factor). Setting $t=1$ recovers $\DT_{t=1}(A_i)=\DT_Q(A_i)$ for each $i\in I$, which admits a separation formula in the form of \eqref{eq: DT F polynomial}. On the other hand, setting $t=0$ only extracts the constant term (with respect to $t$) from $\DT_t(A_i)$. We observe that the consistent scattering diagram $\fD_t$ is the same as the consistent scattering diagram $\fD_Q$ with each $X_j$ substituted by $tX_j$ for all $j\notin I$. Therefore $\DT_t(A_i)$ is the same as $\DT_Q(A_i)$ with $X_j$ substituted by $tX_j$ for all $j\notin I$. Thus, setting $t=0$ is the same as setting $X_j=0$ in the separation formula for $\DT_Q(A_i)$, and hence we can conclude that for $i\in I$, $\DT_{Q'}(A_i)=A_i^{-1}p\left(F_i\big|_{X_j=0 \text{ for all $j\notin I'$}}\right)$. 
\end{proof}

By combining Proposition \ref{prop: DT for subquivers} and Lemma \ref{lem: setting to 0}, we obtain the following corollary, which is useful when computing $F$-polynomials of DT transformations for full subquivers.

\begin{cor}[Truncation for full subquivers]\label{cor: truncation} Let $Q'$ be a full subquiver of a quiver $Q$ that admits a reddening sequence. Suppose $i$ is a vertex in $Q'$ whose $F$-polynomial of $\DT_Q$ is equal to the ideal function of a simply-labeled poset $F_{(P,L)}$. Then the $F$-polynomial of $\DT_{Q'}$ at $i$ is equal to $F_{(P',L')}$, where $(P',L')$ is a simply-labeled poset obtained from $(P,L)$ by deleting the anti-ideal in $P$ generated by elements labeled by $X_j$ with $j$ not in $Q'$.
\end{cor}

Lastly, we would like to consider triangular extensions of quivers. Let $Q_1$ be a quiver with vertex set $I_1$ and exchange matrix $\epsilon_1$, and let $Q_2$ be a quiver with vertex set $I_2$ and exchange matrix $\epsilon_2$. A \emph{triangular extension} of $Q_1$ by $Q_2$ is a quiver $Q$ with vertex set $I=I_1\sqcup I_2$ and exchange matrix $\epsilon=\begin{pmatrix}\epsilon_1 & \delta \\ -\delta & \epsilon_2
\end{pmatrix}$ where all entries of $\delta$ are non-negative. The following theorem describes a way to construct reddening sequences on the triangular extension $Q$ from reddening sequences of $Q_1$ and $Q_2$.

\begin{thm}[{\cite[Theorem 4.5, Remark 4.6]{CL}}]\label{thm: reddening for triangular extension} If $\bk_1$ is a maximal green (resp. reddening) sequence of $Q_1$ and $\bk_2$ is a maximal green (resp. reddening) sequence of $Q_2$, then the concatenation $(\bk_1,
\bk_2)$ is a maximal green (resp. reddening) sequence of the triangular extension of $Q_1$ by $Q_2$. 
\end{thm}

Based on this theorem, we can view the cluster DT transformation $\DT_Q$ as the composition of those of its full subquivers $\DT_{Q_2}\circ \DT_{Q_1}$. Moreover, we prove the following result for $F$-polynomials of cluster DT transformations on triangular extensions.

\begin{prop}\label{prop: F-polynomial for triangular extensions} Suppose $Q$ is a triangular extension of $Q_1$ by $Q_2$. Let $\delta$ be the matrix encoding the arrows from vertices of $Q_1$ to vertices of $Q_2$. Suppose both $Q_1$ and $Q_2$ admit DT transformations, and let $F'_i$ be the $F$-polynomial of $\DT_{Q_1}$ or $\DT_{Q_2}$ at the vertex $i$ (depending on whether $i$ is in $Q_1$ or $Q_2$). Then the $F$-polynomials of the DT transformation on $Q$ are given by
\[
F_i=\left\{\begin{array}{ll} F'_i & \text{if $i$ is in $Q_1$}, \\
F'_i\big|_{X_j\rightsquigarrow X_j\prod_kF'^{\delta_{kj}}_k} & \text{if $i$ is in $Q_2$}.\end{array}\right.
\]
\end{prop}
\begin{proof} Let $\bk_1$ be a reddening sequence of $Q_1$ and let $\bk_2$ be a reddening sequence of $Q_2$. By Theorem \ref{thm: reddening for triangular extension}, the mutation sequence $(\bk_1,\bk_2)$ is a reddening sequence for the triangular extension $Q$. Since $(\bk_1,\bk_2)$ starts with $\bk_1$ and never mutates anywhere on $Q_1$ after finishing $\bk_1$, we know from the $F$-polynomial mutation formula \eqref{eq: F mutation} that $F_i=F'_i$ for all vertices $i$ in $Q_1$. 

On the other hand, if we view vertices of $Q_2$ as frozen during the mutation sequence $\bk_1$, then the arrows between $I_1$ and $I_2$ must change in the same way as the framing arrows in the quiver $\tilde{Q}_1$ with principal coefficients. By the sign coherence theorem \ref{thm: sign coherence}, there will never be arrows created between $I_2$ and the framing vertices in $\tilde{Q}_1$, and hence the $C$-vector of any vertex $j\in I_2$ remains a standard basis vector. After finishing the mutation sequence $\bk_1$, the arrows between $I_1$ and $I_2$ must be reversed such that there are $\delta_{kj}$ many arrows pointing from a vertex $j\in I_2$ to a vertex $k\in I_1$. Thus, according to Theorem \ref{thm: separation formula}, at the moment right after the mutation sequence $\bk_1$, the cluster Poisson variable $X'_j$ of a vertex $j\in I_2$ is expressed in terms of the initial cluster Poisson variables as
\begin{equation}\label{eq: intermediate expression}
X'_j=X_j\prod_{k\in I_1}F'^{\delta_{kj}}_k.
\end{equation}
In particular, since $\delta_{kj}\geq 0$, the right-hand side of \eqref{eq: intermediate expression} is always a polynomial. Finally, the remaining mutation sequence $\bk_2$ will give rise to $F$-polynomials $F'_j$ for each vertex $j\in I_2$ in terms of these intermediate cluster Poisson variables, and to re-express them in terms of the initial cluster Poisson variables, we just need to substitute \eqref{eq: intermediate expression} into $F'_j$'s, and that is what is precisely the formula claimed in this proposition.
\end{proof}

By combining Proposition \ref{prop: F-polynomial for triangular extensions} with Lemma \ref{lem: formula for pointed labeled posets}, we obtain the following corollary, which is useful when computing $F$-polynomials of DT transformations for triangular extensions.

\begin{cor}\label{cor: extension} In addition to the assumptions in Propsition \ref{prop: F-polynomial for triangular extensions}, let us suppose that all $F$-polynomials of $\DT_{Q_1}$ and $\DT_{Q_2}$ are equal to ideal functions of simply-labeled pointed posets $(P(i),L(i))$. Then for any vertex $j$ in $Q_2$, its $F$-polynomial of $\DT_Q$ is equal to the ideal function of a simply-labeled pointed poset obtained by attaching $\delta_{ij}$ many copies of the $(P(i),L(i))$ to any Hasse diagram vertex $X_j$ with an arrow from the lowest point of each $P(i)$ to that vertex labeled by $X_j$.
\end{cor}

\section{\texorpdfstring{$F$-Polynomials of DT on Acyclic Quivers}{}}

\begin{defn} A quiver is said to be \emph{acyclic} if there are no oriented cycles in the quiver. A vertex $i$ in an acyclic quiver $Q$ is said to be a \emph{source} if $\epsilon_{ij}\geq 0$ for all $j\in I$.
\end{defn}


Let $Q$ be an acyclic quiver. We define a simply-labeled poset $(P(i), L(i))$ for each vertex $i$ of $Q$ as follows. The elements in $P(i)$ are homotopic families of oriented paths in $Q$ that end at $i$. We are using homotopy here because we would like to count the composition of a path $\gamma$ from $j$ to $i$ with the idle path at $i$ the same as $\gamma$ itself. The partial order on $P(i)$ is given by inclusion as subpaths up to homotopy. In particular, the idle path at $i$ is considered a path as well, and it is the unique minimal element in $P(i)$. We call $P(i)$ the \emph{ascendant tree} of the vertex $i$. We label the ascendant tree $P(i)$ by labeling each path $\gamma$ from $j$ to $i$ by $X_j$. 


\begin{thm}\label{thm: DT for acyclic quivers} For each vertex $i$ in an acyclic quiver $Q$, the $F$-polynomial of the cluster DT transformation on $Q$ for the vertex $i$ is equal to the ideal function of the simply-labeled pointed poset $(P(i),L(i))$, i.e., 
\[
F_i=F_{(P(i),L(i))}.
\]
\end{thm}
\begin{proof} We will do an induction on the number of vertices of $Q$. If there is only one vertex, then the cluster DT transformation is given by a single mutation at that vertex, which gives an $F$-polynomial $F=1+X=[X]$. Now inductively, let $i$ be a source of $Q$. The full subquiver $Q'$ with the vertex $i$ removed is still acyclic, and the whole quiver can be viewed as a triangular extension of a singleton quiver $\{i\}$ by $Q'$. Now by Proposition \ref{prop: F-polynomial for triangular extensions}, the $F$-polynomial for the source $i$ is just $F=1+X_i=[X_i]$, whereas the $F$-polynomial for any other vertex $j$ can be obtained by taking $F_{(P'(j),L'(j))}$ (which is the $F$-polynomial for $\DT_{Q'}$ by the inductive hypothesis) and then substituting any $X_k$ in it by $X_k(1+X_i)^{\epsilon_{ik}}$. But this is precisely attaching $\epsilon_{ik}$ many branches on top of every vertex labeled by $X_k$ in the Hasse diagram, and we see that the resulting substituted polynomial coincides with the ideal function $F_{(P(j),L(j))}$ as claimed.
\end{proof}

\begin{exmp} Consider the following acyclic quiver. The $F$-polynomials of $\DT$ on this quiver are given on as ideal functions of labeled posets on the right.
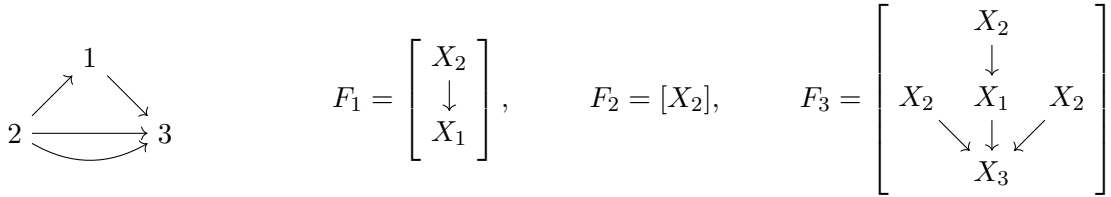
\[
\begin{tikzpicture}[baseline=10]
    \node (1) at (0,1) [] {$1$};
    \node (2) at (-1,0) [] {$2$};
    \node (3) at (1,0) [] {$3$};
    \draw [->] (2) -- (1);
    \draw [->] (1) -- (3);
    \draw [->] (2) -- (3);
    \draw [->] (2) to [bend right] (3);
\end{tikzpicture} \hspace{2cm}
F_1=\left[\begin{tikzpicture}[baseline=10]
   \node (1) at (0,0) [] {$X_1$};
   \node (2) at (0,1) [] {$X_2$};
   \draw [->] (2) -- (1);
\end{tikzpicture}\right], \hspace{1cm}
F_2=[X_2], \hspace{1cm} F_3=\left[\begin{tikzpicture}[baseline=25]
   \node (1) at (0,0) [] {$X_3$};
   \node (2) at (0,1) [] {$X_1$};
   \node (3) at (0,2) [] {$X_2$};
   \node (4) at (-1,1) [] {$X_2$};
   \node (5) at (1,1) [] {$X_2$};
   \draw [->] (2) -- (1);
   \draw [->] (3) -- (2);
   \draw [->] (4) -- (1);
   \draw [->] (5) -- (1);
\end{tikzpicture}\right]
\]
\end{exmp}

\section{\texorpdfstring{$F$-Polynomials of DT on Quivers of Surface Type}{}}

Let $S$ be an oriented surface with a set $M$ of marked points. We call the marked points in the interior of $S$ \emph{punctures} and call the marked points along $\partial S$ \emph{boundary marked points}. In particular, if $\partial S \neq \emptyset$, we require that there is at least one marked point on each connected component of $\partial S$ Fomin, Shapiro, and Thurston \cite{FST} introduce the concept of ``tagged triangulations'', generalizing ideal triangulations of surfaces. 

\begin{defn} On a surface with marked points $(S,M)$, an \emph{arc} is a homotopy equivalence family of unoriented curves $\gamma$ on $S$, connecting one point in $M$ to another point (possibly the same point) in $M$, such that $\gamma$ is not null-homotopic, and there is no self-intersection.
\end{defn}

\begin{defn} There are two possible decorations we can put on either end of an arc: \emph{plain} or \emph{notched}. A \emph{tagged arc} is an arc with (possibly different) decorations at its two ends, such that 
\begin{itemize}
    \item the arc does not bound a once-punctured monogon;
    \item if an arc ends at a boundary marked point, then that end must be plain;
    \item if the arc is a loop, then both ends are tagged the same way.
\end{itemize}
We denote a notched end of an arc by $\begin{tikzpicture}\draw[decoration={markings, mark=at position 0.9 with {\arrow{Rays}}}, postaction={decorate}] (0,0) -- (1,0);  
\end{tikzpicture}$, and denote a plain end of an arc by not decorating it. 
\end{defn}

\begin{defn} Two distinct tagged arcs $\alpha$ and $\beta$ are said to be \emph{compatible} if the followings are satisfied.
\begin{itemize}
    \item The untagged versions of $\alpha$ and $\beta$ are either homotopic or do not intersect in their interior (up to homotopy).
    \item If the untagged versions of $\alpha$ and $\beta$ are not homotopic and they share an end point at $p$, then they have be tagged in the same way near $p$.
    \item If the untagged versions of $\alpha$ and $\beta$ are homotopic, then one of their ends are decorated in the same way and the other end is decorated differently.
\end{itemize}
A maximal set of pairwise compatible tagged arcs is called a \emph{tagged triangulation}.
\end{defn}

Note that it follows from the definition that notched ends can only be present near a puncture, and if both decorations are present near a puncture, then there can only be two tagged arcs with homotopic supports and opposite decorations at that puncture.

There is a canonical way to turn a triangulation $T$ of $(S,M)$ into a tagged triangulation. If there are no self-folded triangles present in $T$, then $T$ is automatically a tagged triangulation with all arcs decorated plainly at both ends. For any self-folded triangle, we replace the arc that bounds a punctured monogon by a tagged arc with a notch at the puncture (see Figure \ref{fig: tagged triangulation convention}).
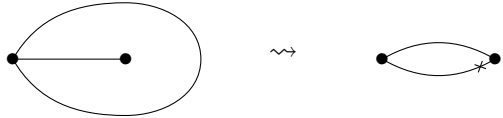
\begin{figure}[H]
    \centering
    \begin{tikzpicture}[baseline=0]
        \draw (0,0) -- (-1.5,0);
        \draw (-1.5,0) to [out=60,in=180] (0,0.75) to [out = 0,in=90] (1,0) to [out=-90,in=0] (0,-0.75) to [out=180,in=-60] (-1.5,0);
        \node at (-1.5,0) [] {$\bullet$};
        \node at (0,0) [] {$\bullet$};
    \end{tikzpicture} \quad \quad $\rightsquigarrow$ \quad \quad 
    \begin{tikzpicture}[baseline=0]
        \draw[decoration={markings, mark=at position 0.9 with {\arrow{Rays}}}, postaction={decorate}] (-1.5,0) to [bend right] (0,0);
        \draw (-1.5,0) to [bend left] (0,0);
        \node at (-1.5,0) [] {$\bullet$};
        \node at (0,0) [] {$\bullet$};
    \end{tikzpicture}
    \caption{Replacing a self-folded triangle by a tagged arc with a notch at the puncture.}
    \label{fig: tagged triangulation convention}
\end{figure}

The introduction of tagged arcs is to allow mutation of a triangulation when 
self-folded triangles are present. In particular, when there is a puncture with both decorations present, then this punctured must be inside a bigon; the mutation at each of the two incident tagged arcs will produce the following two tagged triangulations.
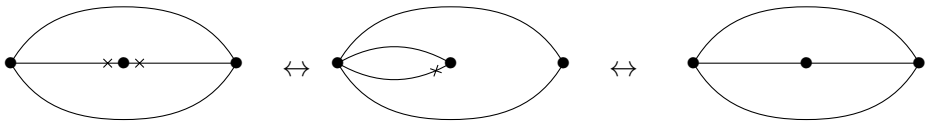
\begin{figure}[H]
    \centering
    \begin{tikzpicture}[baseline=-5]
        \draw[decoration={markings, mark=at position 0.9 with {\arrow{Rays}}}, postaction={decorate}](-1.5,0) -- (0,0);
        \draw[decoration={markings, mark=at position 0.9 with {\arrow{Rays}}}, postaction={decorate}] (1.5,0) -- (0,0);
        \draw (-1.5,0) to [out=60,in=180] (0,0.75) to [out=0,in=120] (1.5,0);
        \draw (-1.5,0) to [out=-60,in=180] (0,-0.75) to [out=0,in=-120] (1.5,0);
        \foreach \i in {-1,0,1}
        {
        \node at (\i*1.5,0) [] {$\bullet$};
        }
    \end{tikzpicture}\quad $\leftrightarrow$
    \begin{tikzpicture}[baseline=-5]
        \draw[decoration={markings, mark=at position 0.9 with {\arrow{Rays}}}, postaction={decorate}] (-1.5,0) to [bend right] (0,0);
        \draw (-1.5,0) to [bend left] (0,0);
        \draw (-1.5,0) to [out=60,in=180] (0,0.75) to [out=0,in=120] (1.5,0);
        \draw (-1.5,0) to [out=-60,in=180] (0,-0.75) to [out=0,in=-120] (1.5,0);
        \foreach \i in {-1,0,1}
        {
        \node at (\i*1.5,0) [] {$\bullet$};
        }
    \end{tikzpicture}\quad $\leftrightarrow$ \quad
    \begin{tikzpicture}[baseline=-5]
        \draw(-1.5,0) -- (1.5,0);
        \draw (-1.5,0) to [out=60,in=180] (0,0.75) to [out=0,in=120] (1.5,0);
        \draw (-1.5,0) to [out=-60,in=180] (0,-0.75) to [out=0,in=-120] (1.5,0);
        \foreach \i in {-1,0,1}
        {
        \node at (\i*1.5,0) [] {$\bullet$};
        }
    \end{tikzpicture}
    \caption{From the middle tagged triangulation, the left picture is obtained by mutating at the plain arc, and the right picture is obtained by mutating at the notched arc.}
    \label{fig: tagged triangulation mutation}
\end{figure}

From a tagged triangulation $T$ of a surface with marked points $(S,M)$, one can produce a quiver $Q_T$. The vertices of $Q_T$ is in bijection with the non-boundary tagged arcs (tagged arcs whose support is not homotopic to a curve on $\partial S$) in $T$. To draw the arrows in $Q_T$, we first observe that $T$ decomposes $S$ into two types of regions: triangles and bigons. Inside each triangule region we draw a counterclockwisely oriented $3$-cycle among the corresponding quiver vertices; inside each bigon, we draw the following pattern on the right among the four quiver vertices involved. If some of the arcs present are boundary arcs, we simply delete that vertex with all arrows incident to it. In the end, we erase a maximal subset of oriented $2$-cycles. Quivers coming from a tagged triangulation are said to be of \emph{surface type}. 
\begin{figure}[H]
    \centering
    \begin{tikzpicture}[baseline=0]
        \draw (-150:1.5) -- (-30:1.5) -- (0,1.5) -- cycle;
        \foreach \i in {1,2,3}
        {
        \node [red] (\i) at (150+\i*120:0.75) [] {$\bullet$}; 
        }
        \draw [->,red] (1) -- (2);
        \draw [->,red] (2) -- (3);
        \draw [->,red] (3) -- (1);
    \end{tikzpicture} \hspace{3cm}
    \begin{tikzpicture}[baseline=0,scale=1.2]
        \draw[decoration={markings, mark=at position 0.9 with {\arrow{Rays}}}, postaction={decorate}] (-1.5,0) to [bend right] (0,0);
        \draw (-1.5,0) to [bend left] (0,0);
        \draw (-1.5,0) to [out=60,in=180] (0,0.75) to [out=0,in=120] (1.5,0);
        \draw (-1.5,0) to [out=-60,in=180] (0,-0.75) to [out=0,in=-120] (1.5,0);
        \node [red] (0) at (-0.75,0.25) [] {$\bullet$};
        \node [red] (1) at (-0.75,-0.25) [] {$\bullet$};
        \node [red] (2) at (0.4,0.7) [] {$\bullet$};
        \node [red] (3) at (0.4,-0.7) [] {$\bullet$};
        \draw [->,red] (2) -- (0);
        \draw [->,red] (2) -- (1);
        \draw [->,red] (0) -- (3);
        \draw [->,red] (1) -- (3);
        \draw [->, red] (3) -- (2);
    \end{tikzpicture}
    \caption{Building blocks for quivers of surface type.}
    \label{fig: quiver of surface type building blocks}
\end{figure}
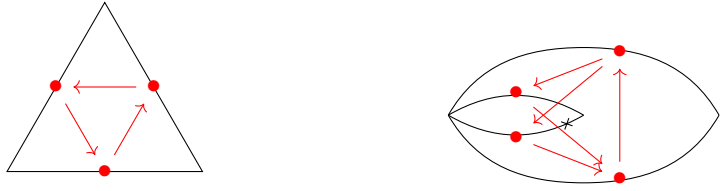

The following result of Goncharov and Shen states that most quivers of surface types admit a cluster DT transformation.

\begin{thm}[{\cite[Theorem 1.3]{GS2}}]\label{thm: DT for surfaces} Suppose $(S,M)$ satisfies the following conditions:
\begin{itemize}
    \item $g(S)+|M|\geq 3$ or $S$ is an annulus with two boundary marked points (one on each boundary);
    \item if $\partial S=\emptyset$, then $|M|>1$.
\end{itemize}
Then any quiver $Q_T$ associated with a tagged triangulation $T$ of $(S,M)$ admits a cluster DT transformation. In particular, any mutation sequence that achieves the following is a reddening sequence:
\begin{itemize}
    \item changing the decoration at each punctured end of a tagged arc to the opposite decoration;
    \item rotating each boundary component of $\partial S$ by one tick (with respect to the induced orientation on $\partial S$).
\end{itemize}
\end{thm}

For the rest of this section, we would like to focus on a family of arcs that we call ``admissible'', and try to express their $F$-polynomials under the cluster DT transformation using ideal functions of labeled posets.

\begin{defn} Let $T$ be a triangulation of $(S,M)$ (possibly with self-folded triangles). A puncture is said to be \emph{admissible} (with respect to the triangulation $T$) if there are no arcs connecting $p$ back to itself and there are at least three arcs incident to $p$. A boundary marked point $s$ is said to be \emph{admissible} (with respect to the triangulation $T$) if there are no arcs connecting $s$ back to itself and $s$ lies on a boundary with at least two boundary marked points. An arc in $T$ is said to be \emph{admissible} (with respect to the triangulation $T$) if it is 
\begin{enumerate}
    \item the unique arc connecting two non-adjacent admissible boundary marked points;
    \item the unique arc connecting an admissible boundary marked point and an admissible puncture;
    \item the unique arc connecting wo admissible punctures.
\end{enumerate}
\end{defn}

For each of the three types of admissible arcs, we define a simply-labeled pointed poset $(P(a),L(a))$ for them using a Hasse diagram as follows.

\begin{enumerate}
    \item Suppose $a$ is the unique arc connecting two non-adjacent admissible boundary marked points $s$ and $t$. Let $b_1,b_2,\dots, b_m$ be the arcs that end at $s$ after $a$ (in the order with respect to the induced orientation on $\partial S$), and let $c_1,c_2,\dots, c_n$ be the arcs that end at $t$ after $a$ (in the order with respect to the induced orientation on $\partial S$). Then the poset $P(a)$ consists of two linearly ordered chains $b_m>b_{m-1}>\cdots >b_1>a$ and $c_n>c_{n-1}>\cdots >c_1>a$, both ending at the same element $a$. We label elements in $P(a)$ by $L(a):i\mapsto X_i$. A picture of the Hasse diagram is shown on the left in Figure \ref{fig: Hasse diagrams for surface quivers}. 

    \item Suppose $a$ is the unique arc connecting an admissible boundary marked point $s$ and an admissible puncture $p$. Let $b_1,b_2,\dots,b_m$ be the arcs that end at $s$ after $a$ (in the order with respect to the induced orientation on $\partial S$), and let $c_1,c_2,\dots, c_n,c_{n+1}=a$ be the arcs that end at $p$ (ordered in the counterclockwise direction). By assumption, we have $n\geq 2$. If $m\geq 1$, then the labeled poset $(P(a),L(a))$ is defined by the Hasse diagram shown in the middle of Figure \ref{fig: Hasse diagrams for surface quivers}. If $m=0$, then the labeled poset $P(a)$ is a linearly ordered chain $c_{n-1}>c_{n-2}>\cdots >c_1>a$, and the labeling is given by $L(a):i\mapsto X_i$.

    \item Suppose $a$ is the unique arc connecting two admissible punctures $p$ and $q$. Let $b_1, b_2, \dots, b_m, b_{m+1}=a$ be the arcs that end at $p$ (ordered in the counterclockwise direction) and let $c_1,c_2,\dots, c_n,c_{n+1}=a$ be the arcs that end at $q$ (ordered in the counterclockwise direction). By assumption, we have $m,n\geq 2$. Then the labeled poset $(P(a),L(a))$ is defined by the Hasse diagram shown on the right in Figure \ref{fig: Hasse diagrams for surface quivers}. 
\end{enumerate}
\begin{figure}[H]
    \centering
    \begin{tikzpicture}
    \node (0) at (0,0.5) [] {$X_a$};
    \node (1) at (-1,1) [] {$X_{b_1}$};
    \node (2) at (-1,2) [] {$\vdots$};
    \node (3) at (-1,3) [] {$X_{b_m}$};
    \node (4) at (1,1) [] {$X_{c_1}$};
    \node (5) at (1,2) [] {$\vdots$};
    \node (6) at (1,3) [] {$X_{c_n}$};
    \draw [->] (3) -- (2);
    \draw [->] (2) -- (1);
    \draw [->] (1) -- (0);
    \draw [->] (6) -- (5);
    \draw [->] (5) -- (4);
    \draw [->] (4) -- (0);
    \end{tikzpicture} \hspace{2cm}
        \begin{tikzpicture}
    \node (0) at (0,0.5) [] {$X_a$};
    \node (1) at (-1,1) [] {$X_{b_1}$};
    \node (2) at (-1,2) [] {$\vdots$};
    \node (3) at (-1,3) [] {$X_{b_m}$};
    \node (4) at (1,1) [] {$X_{c_1}$};
    \node (5) at (1,2) [] {$\vdots$};
    \node (6) at (1,3) [] {$X_{c_n}$};
    \draw [->] (3) -- (2);
    \draw [->] (2) -- (1);
    \draw [->] (1) -- (0);
    \draw [->] (6) -- (5);
    \draw [->] (5) -- (4);
    \draw [->] (4) -- (0);
    \draw [->] (6) -- (1);
    \end{tikzpicture}\hspace{2cm}
     \begin{tikzpicture}
    \node (0) at (0,0.5) [] {$X_a$};
    \node (1) at (-1,1) [] {$X_{b_1}$};
    \node (2) at (-1,2) [] {$\vdots$};
    \node (3) at (-1,3) [] {$X_{b_m}$};
    \node (4) at (1,1) [] {$X_{c_1}$};
    \node (5) at (1,2) [] {$\vdots$};
    \node (6) at (1,3) [] {$X_{c_n}$};
    \node (7) at (0,3.5) [] {$X_a$};
    \draw [->] (3) -- (2);
    \draw [->] (2) -- (1);
    \draw [->] (1) -- (0);
    \draw [->] (6) -- (5);
    \draw [->] (5) -- (4);
    \draw [->] (4) -- (0);
    \draw [->] (6) -- (1);
    \draw [->] (3) -- (4);
    \draw [->] (7) -- (6);
    \draw [->] (7) -- (3);
    \end{tikzpicture}
    \caption{Hasse diagrams for the three types of arcs. Left: unique arc connecting two admissible boundary marked points. Middle: unique arc connecting an admissible boundary marked point and an admissible puncture. Right: unique arc connecting two admissible punctures.}
    \label{fig: Hasse diagrams for surface quivers}
\end{figure}
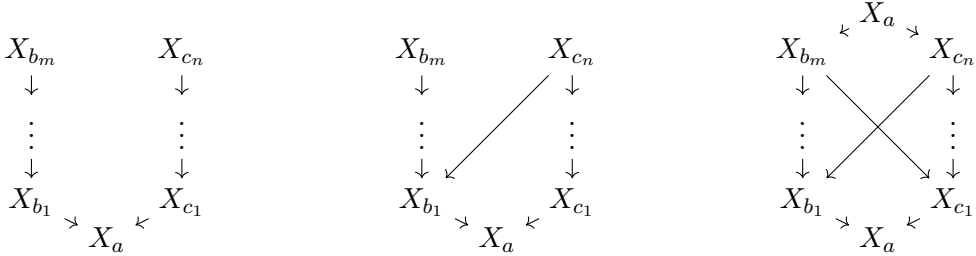

The main goal of this section is to prove the following theorem.

\begin{thm}\label{thm: F poly for DT on surface type quivers} Let $T$ be a triangulation of $(S,M)$. Let $a$ be an admissible arc in $T$. Then the $F$-polynomial of the cluster DT transformation on $Q_T$ for the vertex $a$ is equal to the ideal function of the simply-labeled pointed poset $(P(a),L(a))$, i.e., 
\[
F_a=F_{(P(a),L(a))}
\]
\end{thm}

\begin{proof} We will prove the last case (arcs connecting two punctures) first, and then deduce the first two cases from the last case using Proposition \ref{prop: DT for subquivers}.

Suppose $a$ is an arc connecting two admissible punctures $p$ and $q$. Let us first assume that each of $p$ and $q$ is incident to exactly three arcs (the minimum requirement in Theorem \ref{thm: F poly for DT on surface type quivers}). Then the local picture of the triangulation $T$ near the arc $a$ looks like the left picture in Figure \ref{fig: local picture near an arc connecting two punctures}. If we mutate along the sequence $\bk=(a,b_1,b_2,c_1,c_2,a)$, the local picture changes to the one on the right in Figure \ref{fig: local picture near an arc connecting two punctures}. 
\begin{figure}[H]
    \centering
    \begin{tikzpicture}[scale=1.5,baseline=0]
        \draw (-1,0) -- node [above] {$a$} (1,0) node [right] {$q$};
        \draw (-1,0) node [left] {$p$} -- node [right] {$b_1$} (0,1) -- node [left] {$c_2$} (1,0) -- node [left] {$c_1$} (0,-1) -- node[right] {$b_2$} cycle;
        \draw [dashed] (0,0) ellipse (1.5 and 1);
    \end{tikzpicture}\hspace{3cm}
    \begin{tikzpicture}[scale=1.5,baseline=0]
        \draw [decoration={markings, mark=at position 0.9 with {\arrow{Rays}}, mark=at position 0.17 with {\arrow{Rays}}}, postaction={decorate}] (-1,0) -- node [above] {$a$} (1,0) node [right] {$q$};
        \draw [decoration={markings, mark=at position 0.17 with {\arrow{Rays}}}, postaction={decorate}] (-1,0) node [left] {$p$} -- node [right] {$b_2$} (0,1);
        \draw [decoration={markings, mark=at position 0.9 with {\arrow{Rays}}}, postaction={decorate}] (0,1) -- node [left] {$c_1$} (1,0);
        \draw [decoration={markings, mark=at position 0.17 with {\arrow{Rays}}}, postaction={decorate}] (1,0) -- node [left] {$c_2$} (0,-1);
        \draw [decoration={markings, mark=at position 0.9 with {\arrow{Rays}}}, postaction={decorate}] (0,-1) -- node[right] {$b_1$} (-1,0);
        \draw [dashed] (0,0) ellipse (1.5 and 1);
    \end{tikzpicture}
    \caption{Local picture of the triangulation $T$ before and after the mutation sequence $\mu_\bk$.}
    \label{fig: local picture near an arc connecting two punctures}
\end{figure}
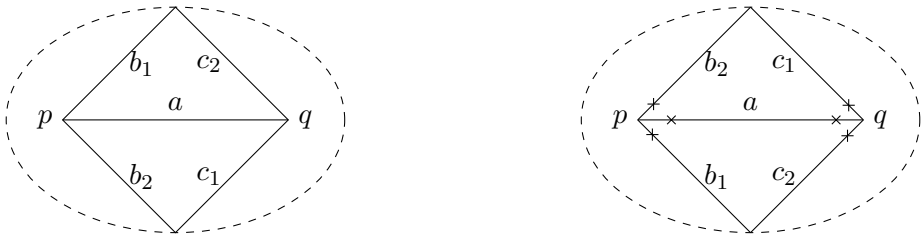

Note that in the right picture in Figure \ref{fig: local picture near an arc connecting two punctures}, all the arcs incident to the punctures $p$ and $q$ are notched near these two punctures, and the arc $a$ is already notched at both ends. By Theorem \ref{thm: DT for surfaces}, the task of changing decorations is done at the two punctures $p$ and $q$, and hence no more mutation is needed at the arc $a$. In other words, there is a reddening sequence for $Q_T$ that starts with the mutation sequence $\bk$, and nowhere in the remaining reddening sequence does the quiver vertex $a$ show up. Now we recall from Equation \ref{eq: F mutation} that $F$-polynomials only change when a mutation occurs at the quiver vertex; thus, the $F$-polynomial $F_a$ of the cluster DT transformation on $Q_T$ for the vertex $a$ is the same as the $F$-polynomial for $a$ after the mutation sequence $\bk$. 

Since $\bk$ only involves five quiver vertices, we can extract this local full subquiver $Q'$ (left picture in Figure \ref{fig: local subquiver for two punctures}) and consider the mutation sequence $\bk$ acting on it. Note that extending the mutation sequence $\bk$ to $\bk'=(a,b_1,b_2,c_1,c_2,a,b_1,c_1)$ does not change the $F$-polynomial $F_a$. But $\bk'$ is actually a reddening sequence on $Q'$, and hence $F_a$ is also the $F$-polynomial of $\DT_{Q'}$ for the vertex $a$. Now to compute the $F$-polynomials of $\DT_{Q'}$, we can make use of Proposition \ref{prop: DT for adjacent clustsers} and Theorem \ref{thm: DT for acyclic quivers}.
\begin{figure}[H]
    \centering
\begin{tikzpicture}[baseline=0]
    \node (0) at (0,0) [] {$a$};
    \node (1) at (1,1) [] {$c_2$};
    \node (2) at (1,-1) [] {$c_1$};
    \node (3) at (-1,-1) [] {$b_2$};
    \node (4) at (-1,1) [] {$b_1$};
    \draw [<-] (0) -- (2);
    \draw [<-] (2) -- (1);
    \draw [<-] (1) -- (0);
    \draw [<-] (0) -- (4);
    \draw [<-] (4) -- (3);
    \draw [<-] (3) -- (0);
    \draw [<-] (4) -- (1);
    \draw [<-] (2) -- (3);
    \end{tikzpicture} \hspace{3cm}
    \begin{tikzpicture}[baseline=0]
    \node (0) at (0,0) [] {$a$};
    \node (1) at (1,1) [] {$c_2$};
    \node (2) at (1,-1) [] {$c_1$};
    \node (3) at (-1,-1) [] {$b_2$};
    \node (4) at (-1,1) [] {$b_1$};
    \draw [->] (0) -- (2);
    \draw [->] (1) -- (0);
    \draw [->] (0) -- (4);
    \draw [->] (3) -- (0);
    \end{tikzpicture}
    \caption{Local full subquiver $Q'$ (left) and its mutated quiver $\mu_a(Q')$ (right).}
    \label{fig: local subquiver for two punctures}
\end{figure}
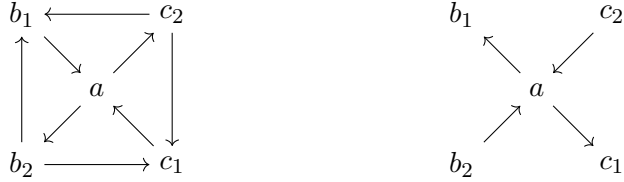

Based on Theorem \ref{thm: DT for acyclic quivers}, we list the $F$-polynomials of $\DT_{\mu_a(Q')}$ as below. Here we use primed notations for the $F$-polynomials and the cluster Poisson variables for the quiver $\mu_a(Q')$ to distinguish from those associated with the full subquiver $Q'$. We also make use of Lemma \ref{lem: formula for pointed labeled posets} to simplify the formulas.
\[
F'_a=1+X'_a(1+X'_{b_2})(1+X'_{c_2}), \quad\quad F'_{b_1}=1+X'_{b_1}F'_a, \quad\quad F'_{b_2}=1+X'_{b_2}, 
\]
\[
F'_{c_1}=1+X'_{c_1}F'_a, \quad\quad F'_{c_2}=1+X'_{c_2}.
\]
According to Equation \ref{eq: F mutation}, 
\begin{align*}
\mu_a(F'_a)=&\frac{(1+X'_{b_1}F'_a)(1+X'_{c_1}F'_a)+X'_a(1+X'_{b_2})(1+X'_{c_2})}{F'_a}\\
=&\frac{F'_a+(X'_{b_1}+X'_{c_1})F'_a+X'_{b_1}X'_{c_1}F'^2_a}{F'_a}\\
=&1+X'_{b_1}+X'_{c_1}+X'_{b_1}X'_{c_1}F'_a.
\end{align*}
Next, according to Proposition \ref{prop: DT for adjacent clustsers}, we need to do the following change of variables:
\[
X'_a=X_a^{-1}, \quad \quad X'_{b_1}=\frac{X_aX_{b_1}}{1+X_a}, \quad \quad X'_{b_2}=X_{b_2}(1+X_a), \quad \quad X'_{c_1}=\frac{X_aX_{c_1}}{1+X_a}, \quad \quad X'_{c_2}=X_{c_2}(1+X_a),
\]
and then multiply the outcome by $(1+X_a)$. One can check that the result is equal to the ideal function of the labeled poset on the left of Figure \ref{fig: labeled poset for an arc connecting two punctures}.
\begin{figure}[H]
\centering
\begin{tikzpicture}[baseline=20]
    \node (0)  at (0,0) [] {$X_a$};
    \node (1) at (-1,1) [] {$X_{b_1}$};
     \node (2) at (-1,2) [] {$X_{b_2}$};
      \node (3) at (1,1) [] {$X_{c_1}$};
       \node (4) at (1,2) [] {$X_{c_2}$};
        \node (5) at (0,3) [] {$X_a$};
        \draw [->] (1) -- (0);
        \draw [->] (2) -- (1);
        \draw [->] (3) -- (0);
        \draw [->] (4) -- (3);
        \draw [->] (5) -- (4);
        \draw [->] (5) -- (2);
        \draw [->] (4) -- (1);
        \draw [->] (2) -- (3);
\end{tikzpicture}
\hspace{3cm}
\begin{tikzpicture}[baseline=40,scale=0.9]
    \node (0)  at (0,0.25) [] {$X'_a$};
    \node (1) at (-1,1) [] {$X'_{b_1}$};
     \node (2) at (-1,2) [] {$X'_{b_3}$};
     \node (3) at (-1,4) [] {$X'_{b_m}$};
     \node (8) at (-1,3) [] {$\vdots$};
     \node (9) at (1,3) [] {$\vdots$};
      \node (4) at (1,1) [] {$X'_{c_1}$};
       \node (5) at (1,2) [] {$X'_{c_2}$};
       \node (6) at (1,4) [] {$X'_{c_n}$};
        \node (7) at (0,4.75) [] {$X'_a$};
        \draw [->] (1) -- (0);
        \draw [->] (2) -- (1);
        \draw [->] (4) -- (0);
        \draw [->] (5) -- (4);
        \draw [->] (3) -- (4);
        \draw [->] (6) -- (1);
        \draw [->] (7) -- (6);
        \draw [->] (7) -- (3);
        \draw [->] (8) -- (2);
        \draw [->] (3) -- (8);
        \draw [->] (9) -- (5);
        \draw [->] (6) -- (9);
\end{tikzpicture}
\caption{Labeled posets for admissible arcs connecting two punctures: base case (left) and the inductive hypothesis (right).}
\label{fig: labeled poset for an arc connecting two punctures}
\end{figure}
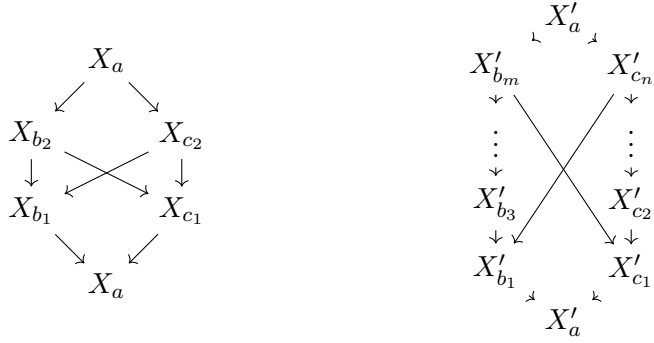

For the general case of arcs connecting two punctures, we will do an induction on the sum $m+n$ of the numbers of arcs incident to the admissible punctures $p$ and $q$. The base case with $m+n=3+3=6$ has been taken care of above. Suppose $a$ is an arc connecting the two punctures $p$ and $q$, with $m+n>6$. Since $m+n>6$ and $m,n\geq 3$, we can assume without loss of generality that $m>3$. Now take the arc $b_2$ and flip (mutate) it. The new arc $b'_2$ connects the other endpoints of $b_1$ and $b_3$, which are not $q$ because $a$ is assumed to be the only arc shared by $p$ and $q$. This decreases the total number $m+n$ by $1$. By the induction hypothesis, the $F$-polynomial $F'_a$ is given as the ideal function of the labeled poset on the right of Figure \ref{fig: labeled poset for an arc connecting two punctures}. To get back the desired $F$-polynomial $F_a$, we just need to apply Proposition \ref{prop: DT for adjacent clustsers} to the mutation at $b'_2$. The only cluster Poisson variables affected are 
\begin{equation}\label{eq2}
X'_{b_1}=X_{b_1}(1+X_{b_2}^{-1})^{-1} \quad \quad \text{and} \quad \quad X'_{b_3}=X_{b_3}(1+X_{b_2}),
\end{equation}
while all other cluster Poisson variables in the labeled poset are simply substituted as $X'_i=X_i$. Now we apply Lemma \ref{lem: insertion} in the reverse direction to conclude that the change of variables in \eqref{eq2} is the same as replacing $X'_{b_1}$ by $X_{b_1}$, replacing $X'_{b_3}$ by $X_{b_3}$, and then inserting $X_{b_2}$ in between $X_{b_1}$ and $X_{b_3}$. This concludes the induction.

\bigskip

For the remaining two types of admissible arcs, we can reduce them to the case of an arc connecting two admissible punctures by adding a ``collar'' along the boundary. Suppose $C$ is a connected component of $\partial S$ with at least two marked points. The \emph{collar extension} of the surface $(S,M)$ along $C$ glues a copy of $C\times [0,1]$ to $S$ by identifying $C$ with $C\times \{0\}$, and places the same number of boundary marked points along $C\times \{1\}$. We can triangulate the collar strip as shown in Figure \ref{fig: collar strip}, which extends any triangulation $T$ of $(S,M)$ to a triangulation of the collar extension.
\begin{figure}[H]
    \begin{tikzpicture}
        \draw (0,0) -- (7,0);
        \draw (0,1) -- (7,1);
        \foreach \i in {1,2,3,5,6,7}
        {
        \draw (\i-0.5,0) -- (\i-0.5,1);
        \draw (\i-0.5,1) -- (\i,0.5);
        \draw (\i-0.5,0) -- (\i-1,0.5);
        }
        \node at (3.5,0.5) [] {$\cdots$};
        \draw [dashed] (0,-0.25) -- (0,1.25);
        \draw [dashed] (7,-0.25) -- (7,1.25);
    \end{tikzpicture}
    \caption{Triangulation of the collar strip}
    \label{fig: collar strip}
\end{figure}
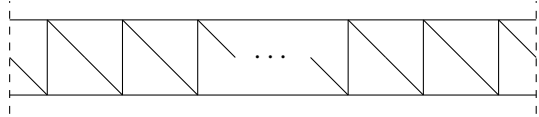

We observe that if $s$ is an admissible boundary marked point with respect to a triangulation $T$ of $(S,M)$, then it becomes an admissible puncture on the collar extension with the induced triangulation $T'$. If $a$ is the unique arc connecting an admissible boundary marked point $s$ to an admissible puncture $p$, we add a collar extension to the boundary component containing $s$. If $a$ is the unique arc connecting two non-adjacent admissible boundary marked points, we add a collar extension to each of the boundary components containing these two boundary marked points (they may also lie on the same boundary component, in which case we only add one collar extension). Moreover, the quiver $Q_T$ for the original triangulation is now a full subquiver of the quiver $Q_{T'}$: the extra vertices in $Q_{T'}$ all come from the non-boundary arcs inside the collar strip. By Theorem \ref{thm: DT for surfaces}, if $Q_T$ has a cluster DT transformation, so does $Q_{T'}$.

Since $Q_T$ is a full subquiver of $Q'_T$, we can make use of Proposition \ref{prop: DT for subquivers} and Lemma \ref{lem: setting to 0} to write down the $F$-polynomial $F_a$ of the cluster DT transformation on $Q_T$ at $a$. By treating $a$ as a unique arc connecting two admissible punctures in the collar extension, we obtain an $F$-polynomial $F'_a$ that is equal to the ideal function of a labeled poset similar to the right picture in Figure \ref{fig: Hasse diagrams for surface quivers}. To reduce to $F_a$, we need to set some of the cluster Poisson variables to $0$. In particular, the boundary arcs on $\partial S$ right after the admissible boundary marked points are among the vanishing cluster Poisson variables, and they generate an anti-ideal that contains all the labelings that will be set to $0$. Thus, we can use Lemma \ref{lem: setting to 0} to slice off the anti-ideal portion of the Hasse diagram and obtain the $F$-polynomial $F_a$, as shown in the first two pictures in Figure \ref{fig: Hasse diagrams for surface quivers}.
\begin{figure}[H]
    \centering
    \begin{tikzpicture}
        \path [fill=lightgray] (-2,0) arc (-180:0:2 and 0.5) -- (2,1) arc (0:-180:2 and 0.5) --cycle;
        \draw (-2,0) arc (-180:0:2 and 0.5);
        \draw (-2,1) arc (-180:0:2 and 0.5);
        \draw (0,-0.5) -- (0,-1.5) node [below] {$a$};
        \draw (0,-0.5) -- (0.5,-1.25) node [below right] {$b_1$};
        \draw (0,-0.5) -- (1,-1.1) node [right] {$b_2$};
        \node at (1,-0.7) [] {$\vdots$};
        \draw (0,-0.5) -- (-0.5,-1.25);
        \draw (0,-0.5) -- (-1,-1.1);
        \node at (-1,-0.7) [] {$\vdots$};
        \draw [red] (0,-0.5) to [out=45, in=-120] (1,0.57);
        \draw [red] (0,-0.5) arc (-90:-120:2 and 0.5) node [black] {$\bullet$};
        \draw [red] (0,-0.5) arc (-90:-60:2 and 0.5) node [black] {$\bullet$};
        \draw [red] (0,-0.5) to (0,0.5);
        \node at (1,0.57) [] {$\bullet$};
        \node at (0,0.5) [] {$\bullet$};
        \node at (0,-0.5) [] {$\bullet$};
    \end{tikzpicture}
    \caption{We color the cluster Poisson variables on $Q_{T'}$ that need to be set to zero in red.}
    \label{fig:my_label}
\end{figure}
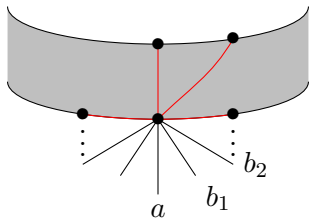

Lastly, if there are no arcs in $T$ that are incident to $s$ after $a$ (with respect to the orientation on $\partial S$), then $b_1$ itself needs to be set to $0$. If the other end of $a$ is an admissible puncture, then the right picture in Figure \ref{fig: Hasse diagrams for surface quivers} collapses down to a linearly ordered chain of the form $X_a\leftarrow X_{c_1}\leftarrow \cdots \leftarrow X_{c_{n-1}}$. \end{proof}

\begin{rmk} When self-folded triangles are present (which means that there is an arc going from a marked point back to itself, making the marked point not admissible), one may find $F$-polynomials of the cluster DT transformation that cannot possibly be ideal functions of a simply-labeled pointed poset. For example, consider the twice punctured disk with a tagged triangulation shown in Figure \ref{fig: twice punctured disk}. The quiver is given on the right of the same figure. One can check that the mutation sequence $(5,1,3,4,2,5)$ is a maximal green sequence, and by going through that sequence, one finds that the $F$-polynomial of DT for the vertex $2$ is
\begin{align*}
F_2=&1 + X_2+ X_1X_2+ X_2X_3 + X_2X_4+ X_1X_2X_3+ X_1X_2X_4 + X_2X_3X_4+ X_1X_2X_3X_4\\
 &   + X_1X_2X_3X_5  + X_1X_2X_4X_5  + X_2X_3X_4X_5  + 2X_1X_2X_3X_4X_5+X_1X_2X_3X_4X_5^2.
\end{align*}
However, there is no simply-labeled pointed poset that has $F_2$ as its ideal function.
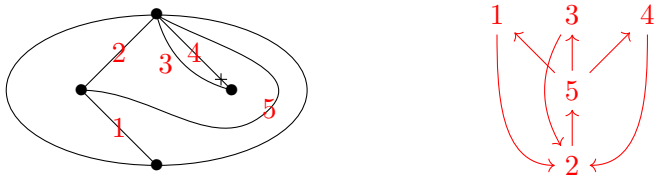
\begin{figure}[H]
    \centering
    \begin{tikzpicture}[baseline=0]
        \draw (0,0) ellipse (2 and 1);
        \draw (0,-1) -- node [red] {$1$} (-1,0) -- node [red] {$2$} (0,1);
        \draw [decoration={markings, mark=at position 0.9 with {\arrow{Rays}}}, postaction={decorate}] (0,1) -- node [red] {$4$} (1,0);
        \draw (0,1) to [bend right] node [red,left] {$3$} (1,0);
        \draw (-1,0) to [out=0,in=-135] (1.5,-0.25)  node [red] {$5$} to [out=45,in=-30] (0,1);
        \foreach \i in {-1,1}
        {
        \node at (\i,0) [] {$\bullet$};
        \node at (0,\i) [] {$\bullet$};
        }
    \end{tikzpicture} \hspace{2cm}
    \begin{tikzpicture}[baseline=0]
    \node [red] (5) at (0,0) [] {$5$};
    \node [red] (1) at (-1,1) [] {$1$};
    \node [red] (2) at (0,-1) [] {$2$};
    \node [red] (3) at (0,1) [] {$3$};
    \node [red] (4) at (1,1) [] {$4$};
    \draw [->, red] (2) -- (5);
    \foreach \i in {1,4}
    {
    \draw [->, red] (5)  -- (\i);
    \draw [->, red] (\i) to [out=-90, in=180*\i] (2);
    }
    \draw [->, red] (5) -- (3);
    \draw [->, red] (3) to [bend right] (2);
    \end{tikzpicture}
    \caption{Example of an $F$-polynomial of DT that is not an ideal function.}
    \label{fig: twice punctured disk}
\end{figure}
\end{rmk}

\section{\texorpdfstring{$F$-Polynomial of DT on Quivers associated with a Triple of Flags}{}}

In \cite{FGteich}, Fock and Goncharov consider the configuration space of a triple of flags in $\bC^n$ and construct a quiver $Q_n$ for the cluster structure on this configuration space. This quiver contains quivers for many interesting families of cluster varieties as subquivers (e.g., double Bruhat cells, Grassmannians). Thus, by Proposition \ref{prop: DT for subquivers}, knowing the $F$-polynomials of DT on $Q_n$ can inform us about the $F$-polynomials of DT on many more families of cluster varieties. In this section, we will consider the quiver $Q_n$ and show that its cluster DT $F$-polynoimals are also given by ideal functions of labeled posets.

Let us first describe the quiver $Q_n$. Consider the simplex 
\[
S=\{(x,y,z) \mid x,y,z\geq 0, x+y+z=n-1\}
\]
and consider the set of lattice points on $S$. These lattice points lie on lines that are parallel to coordinate hyperplanes. These lines in turn triangulate $S$ into smaller triangles. Among these smaller triangles, some are pointing in the same direction as $S$, and some are not. For those that are pointing in the same direction as $S$, we draw a little bipartite tripod of the form $\begin{tikzpicture}[baseline=-2,scale=0.5]
\foreach \i in {0,1,2}
{
\draw (0,0) -- (90+\i*120:0.5);
\draw [fill=white] (90+\i*120:0.5) circle [radius=0.1];
\draw [fill=black] (0,0) circle [radius=0.1];
}
\end{tikzpicture}$ inside. This produces a bipartite graph $\bW_n$ on $S$. This refined triangulation of $S$ is called an $n$-triangulation in \cite{FGteich,GS2} and the bipartite graph $\bW_n$ is called an \emph{$\SL_n$-ideal web} in \cite{Gon}. See Figure \ref{fig: bipartite graph} for an example when $n=5$.
\begin{figure}[H]
    \centering
    \begin{tikzpicture}
        \draw (-150:2) -- (-30:2) -- (90:2) -- cycle;
        \draw (-0.866*1.5,-0.25) -- (0.866*1.5,-0.25) -- (0.866,-1) -- (-0.866*0.5,1.25)--(0.866*0.5,1.25) -- (-0.866,-1) -- cycle;
        \draw (-0.866,0.5) -- (0.866,0.5) -- (0,-1) -- cycle;
    \end{tikzpicture}
    \hspace{2cm}
    \begin{tikzpicture}
        \draw [lightgray](-150:2) -- (-30:2) -- (90:2) -- cycle;
        \draw[lightgray] (-0.866*1.5,-0.25) -- (0.866*1.5,-0.25) -- (0.866,-1) -- (-0.866*0.5,1.25)--(0.866*0.5,1.25) -- (-0.866,-1) -- cycle;
        \draw[lightgray] (-0.866,0.5) -- (0.866,0.5) -- (0,-1) -- cycle;
        \draw (0,2) -- (0,1.5);
        \draw (0,1.5) -- (-0.866*0.5,1.25);
        \draw (0,1.5) -- (0.866*0.5,1.25);
        \draw [fill=white] (0,2) circle [radius=0.1];
        \draw [fill=black] (0,1.5) circle [radius=0.1];
        \foreach \i in {0,1}
        {
        \draw (-0.866*0.5+0.866*\i,1.25) -- (-0.866*0.5+0.866*\i,0.75);
        \draw (-0.866*0.5+0.866*\i,0.75) -- (-0.866+0.866*\i,0.5);
        \draw (-0.866*0.5+0.866*\i,0.75) -- (0.866*\i,0.5);
        \draw [fill=white] (-0.866*0.5+0.866*\i,1.25) circle [radius=0.1];
        \draw [fill=black] (-0.866*0.5+0.866*\i,0.75) circle [radius=0.1];
        }
        \foreach \i in {0,1,2}
        {
        \draw (-0.866+0.866*\i,0.5) -- (-0.866+0.866*\i,0);
        \draw (-0.866+0.866*\i,0) -- (-0.866*1.5+0.866*\i,-0.25);
        \draw (-0.866+0.866*\i,0) -- (-0.866*0.5+0.866*\i,-0.25);
        \draw [fill=white] (-0.866+0.866*\i,0.5) circle [radius=0.1];
        \draw [fill=black] (-0.866+0.866*\i,0) circle [radius=0.1];
        }
        \foreach \i in {0,...,3}
        {
        \draw (-0.866*1.5+0.866*\i,-0.25) -- (-0.866*1.5+0.866*\i,-0.75);
        \draw (-0.866*1.5+0.866*\i,-0.75) -- (-0.866*2+0.866*\i,-1);
        \draw (-0.866*1.5+0.866*\i,-0.75) -- (-0.866+0.866*\i,-1);
        \draw [fill=white] (-0.866*1.5+0.866*\i,-0.25) circle [radius=0.1];
        \draw [fill=black] (-0.866*1.5+0.866*\i,-0.75) circle [radius=0.1];
        }
        \foreach \i in {0,...,4}
        {
        \draw [fill=white] (-2*0.866+0.866*\i,-1) circle [radius=0.1];
        }
    \end{tikzpicture} \hspace{2cm}
    \begin{tikzpicture}
        \draw [lightgray](-150:2) -- (-30:2) -- (90:2) -- cycle;
        \draw [lightgray] (0,2) -- (0,1.5);
        \draw [lightgray] (0,1.5) -- (-0.866*0.5,1.25);
        \draw [lightgray] (0,1.5) -- (0.866*0.5,1.25);
        \draw [lightgray,fill=white] (0,2) circle [radius=0.1];
        \draw [lightgray,fill=lightgray] (0,1.5) circle [radius=0.1];
        \node [red] (00) at (0,1) [] {$\bullet$};
        \foreach \i in {0,1}
        {
        \draw [lightgray] (-0.866*0.5+0.866*\i,1.25) -- (-0.866*0.5+0.866*\i,0.75);
        \draw [lightgray] (-0.866*0.5+0.866*\i,0.75) -- (-0.866+0.866*\i,0.5);
        \draw [lightgray] (-0.866*0.5+0.866*\i,0.75) -- (0.866*\i,0.5);
        \draw [lightgray, fill=white] (-0.866*0.5+0.866*\i,1.25) circle [radius=0.1];
        \draw [lightgray, fill=lightgray] (-0.866*0.5+0.866*\i,0.75) circle [radius=0.1];
        \node [red] (1\i) at (-0.866*0.5+0.866*\i,0.25) [] {$\bullet$};
        }
        \foreach \i in {0,1,2}
        {
        \draw [lightgray] (-0.866+0.866*\i,0.5) -- (-0.866+0.866*\i,0);
        \draw [lightgray] (-0.866+0.866*\i,0) -- (-0.866*1.5+0.866*\i,-0.25);
        \draw [lightgray] (-0.866+0.866*\i,0) -- (-0.866*0.5+0.866*\i,-0.25);
        \draw [lightgray, fill=white] (-0.866+0.866*\i,0.5) circle [radius=0.1];
        \draw [lightgray, fill=lightgray] (-0.866+0.866*\i,0) circle [radius=0.1];
        \node [red] (2\i) at (-0.866+0.866*\i,-0.5) [] {$\bullet$};
        }
        \foreach \i in {0,...,3}
        {
        \draw [lightgray] (-0.866*1.5+0.866*\i,-0.25) -- (-0.866*1.5+0.866*\i,-0.75);
        \draw [lightgray] (-0.866*1.5+0.866*\i,-0.75) -- (-0.866*2+0.866*\i,-1);
        \draw [lightgray] (-0.866*1.5+0.866*\i,-0.75) -- (-0.866+0.866*\i,-1);
        \draw [lightgray, fill=white] (-0.866*1.5+0.866*\i,-0.25) circle [radius=0.1];
        \draw [lightgray, fill=lightgray] (-0.866*1.5+0.866*\i,-0.75) circle [radius=0.1];
        }
        \foreach \i in {0,...,4}
        {
        \draw [lightgray, fill=white] (-2*0.866+0.866*\i,-1) circle [radius=0.1];
        }
        \draw [red,->] (11) -- (10);
        \draw [red,->] (22) -- (21);
        \draw [red,->] (21) -- (20);
        \draw [red,->] (10) -- (00);
        \draw [red,->] (20) -- (10);
        \draw [red,->] (21) -- (11);
        \draw [red,->] (00) -- (11);
        \draw [red,->] (11) -- (22);
        \draw [red,->] (10) -- (21);
    \end{tikzpicture}
    \caption{The $5$-triangulation of a simplex, the corresponding $\SL_5$-ideal web, and the associated quiver $Q_5$.}
    \label{fig: bipartite graph}
\end{figure}
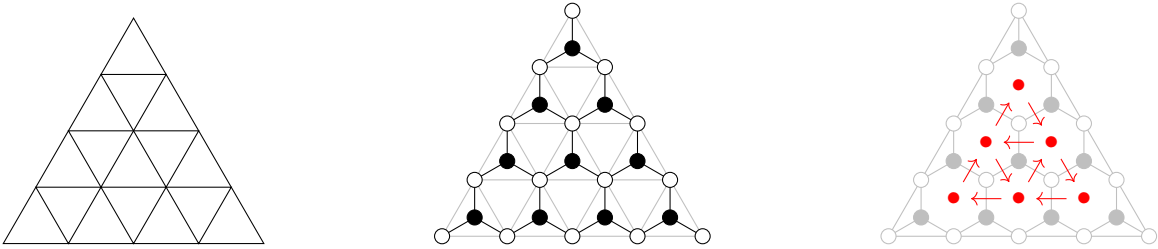
To obtain the quiver $Q_n$, we put a quiver vertex inside each bounded face of the $\SL_n$-ideal web, and draw a clockwise $3$-cycle around each white vertex that is not on $\partial S$. Note that the vertices of the quiver $Q_n$ are in bijection with triples of non-negative integers $(a,b,c)$ such that $a+b+c=n-3$. See Figure \ref{fig: bipartite graph} for an example when $n=5$.

\begin{thm}[{\cite[Theorem 1.3]{GS2}}]\label{thm: DT for triple of flags} $Q_n$ quivers all admit cluster DT transformations. 
\end{thm} 

In \cite{Wengdb}, Weng proves that the cluster DT transformation on $T\backslash \GL_n^{e,w_0}/T$ (which has $Q_n$ as an initial quiver) can be computed geometrically as the composition of the following three morphisms:
\begin{enumerate}
    \item the amalgamation morphism $\spec(\mathcal{X})\rightarrow \GL_n^{e,w_0}$ defined by Fock and Goncharov \cite{FGamalgamation};
    \item the morphism of taking generalized minors $\GL_n^{e,w_0}\rightarrow \spec(\mathcal{A})$ defined by Berenstein, Fomin, and Zelevinsky \cite{BFZ};
    \item the cluster ensemble morphism $\spec(\mathcal{A})\rightarrow \spec(\mathcal{X})$ defined by Fock and Goncharov \cite{FGensemble}, which we have already seen in \eqref{eq: p map}.
\end{enumerate}

By using the ideal web, we can combinatorialize these morphisms above. First, the composition of the amalgamation morphism and the morphism of taking generalized minors can be computed combinatorially using Postnikov's boundary measurement map \cite{Pos}. In detail, we first add an external edge to each white vertex on the left side and on the right side of the ideal web $\bW_n$ (in particular, the top white vertex gets two external edges, one on each side). We label these two sets of external edges by $1,2,\dots, n$ from top to bottom. Inside the extended bipartite graph $\tilde{\bW}_n$, we orient the non-vertical edges from left to right and orient the vertical edges from top to bottom. We hence call the external edges on the left \emph{sources} and the external edges on the right \emph{sinks}. An example of this orientation when $n=5$ is drawn in the left picture of Figure \ref{fig: perfect orientation}. Let $\gamma$ be a path in $\tilde{\bW}_n$ going from a source $i$ to a sink $j$. A bounded face $f$ in $\tilde{\bW}_n$ is said to be \emph{dominated} by $\gamma$ if it lies on the right side of $\gamma$ with respect to its orientation. We denote the set of bounded faces dominated by $\gamma$ by $\hat{\gamma}$. For any two subsets $I,J\subset \{1,2,\dots, n\}$ of the same size, we consider families of $|I|$ many pairwise disjoint paths that go from each element in $I$ to each element in $J$, and collect these families into a set $C(I,J)$; in other words, elements in $C(I,J)$ are of the form $\Gamma=\{\gamma_1,\gamma_2,\dots, \gamma_{|I|}\}$, where each $\gamma_i$ is a path from an element in $I$ to an element in $J$. 

We associate a formal variable $X_f$ with each bounded face $f$ in the ideal web $\bW_n$. Postnikov's boundary measurement map produces a polynomial $M_I^J$ in $X_f$'s for each pair of equal-size subsets $I,J\subset\{1,2,\dots, n\}$ \footnote{In general, Postnikov's boundary measurement map produces a rational function in $X_f$'s, and its actual definition has an extra winding number factor than \eqref{eq: boundary measurement map}; however, since the orientation we put on $\bW_n$ is acyclic, we can ignore this winding number factor and the result is guaranteed to be a polynomial.}:
\begin{equation}\label{eq: boundary measurement map} M_I^J=\sum_{\Gamma\in C(I,J)}\prod_{\gamma\in \Gamma}\prod_{f\in \hat{\gamma}}X_f.
\end{equation}

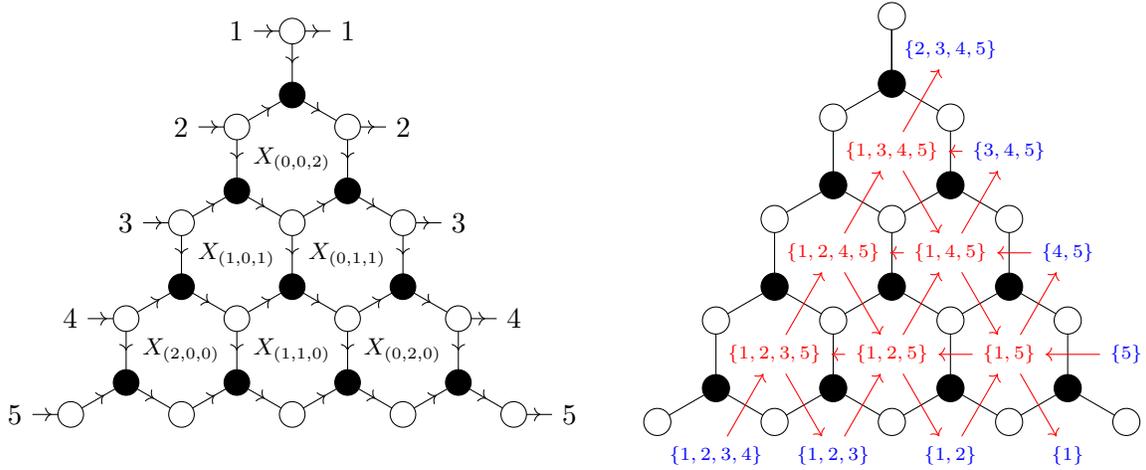
\begin{figure}[H]
    \centering
    \begin{tikzpicture}[scale=1.7,baseline=0]
    \foreach \i in {1,...,5}
    {
    \draw[decoration={markings, mark=at position 0.5 with {\arrow{>}}}, postaction={decorate}] (-0.3-0.866*\i*0.5+0.866*0.5,2.75-0.75*\i) node [left] {$\i$} -- (-0.866*\i*0.5+0.866*0.5,2.75-0.75*\i);
    \draw[decoration={markings, mark=at position 0.7 with {\arrow{>}}}, postaction={decorate}] (0+0.866*\i*0.5-0.866*0.5,2.75-0.75*\i) -- (0.3+0.866*\i*0.5-0.866*0.5,2.75-0.75*\i) node[right] {$\i$};
    }
        \draw[decoration={markings, mark=at position 0.5 with {\arrow{>}}}, postaction={decorate}] (0,2) -- (0,1.5);
        \draw[decoration={markings, mark=at position 0.5 with {\arrow{<}}}, postaction={decorate}] (0,1.5) -- (-0.866*0.5,1.25);
        \draw[decoration={markings, mark=at position 0.5 with {\arrow{>}}}, postaction={decorate}] (0,1.5) -- (0.866*0.5,1.25);
        \draw [fill=white] (0,2) circle [radius=0.1];
        \draw [fill=black] (0,1.5) circle [radius=0.1];
        \foreach \i in {0,1}
        {
        \draw[decoration={markings, mark=at position 0.5 with {\arrow{>}}}, postaction={decorate}](-0.866*0.5+0.866*\i,1.25) -- (-0.866*0.5+0.866*\i,0.75);
        \draw[decoration={markings, mark=at position 0.5 with {\arrow{<}}}, postaction={decorate}] (-0.866*0.5+0.866*\i,0.75) -- (-0.866+0.866*\i,0.5);
        \draw[decoration={markings, mark=at position 0.5 with {\arrow{>}}}, postaction={decorate}] (-0.866*0.5+0.866*\i,0.75) -- (0.866*\i,0.5);
        \draw [fill=white] (-0.866*0.5+0.866*\i,1.25) circle [radius=0.1];
        \draw [fill=black] (-0.866*0.5+0.866*\i,0.75) circle [radius=0.1];
        }
        \foreach \i in {0,1,2}
        {
        \draw[decoration={markings, mark=at position 0.5 with {\arrow{>}}}, postaction={decorate}] (-0.866+0.866*\i,0.5) -- (-0.866+0.866*\i,0);
        \draw[decoration={markings, mark=at position 0.5 with {\arrow{<}}}, postaction={decorate}] (-0.866+0.866*\i,0) -- (-0.866*1.5+0.866*\i,-0.25);
        \draw [decoration={markings, mark=at position 0.5 with {\arrow{>}}}, postaction={decorate}](-0.866+0.866*\i,0) -- (-0.866*0.5+0.866*\i,-0.25);
        \draw [fill=white] (-0.866+0.866*\i,0.5) circle [radius=0.1];
        \draw [fill=black] (-0.866+0.866*\i,0) circle [radius=0.1];
        }
        \foreach \i in {0,...,3}
        {
        \draw [decoration={markings, mark=at position 0.5 with {\arrow{>}}}, postaction={decorate}](-0.866*1.5+0.866*\i,-0.25) -- (-0.866*1.5+0.866*\i,-0.75);
        \draw [decoration={markings, mark=at position 0.5 with {\arrow{<}}}, postaction={decorate}](-0.866*1.5+0.866*\i,-0.75) -- (-0.866*2+0.866*\i,-1);
        \draw [decoration={markings, mark=at position 0.5 with {\arrow{>}}}, postaction={decorate}](-0.866*1.5+0.866*\i,-0.75) -- (-0.866+0.866*\i,-1);
        \draw [fill=white] (-0.866*1.5+0.866*\i,-0.25) circle [radius=0.1];
        \draw [fill=black] (-0.866*1.5+0.866*\i,-0.75) circle [radius=0.1];
        }
        \foreach \i in {0,...,4}
        {
        \draw [fill=white] (-2*0.866+0.866*\i,-1) circle [radius=0.1];
        }
        \node at (0,1) [] {\footnotesize{$X_{(0,0,2)}$}};
        \node at (-0.866*0.5,0.25) [] {\footnotesize{$X_{(1,0,1)}$}};
        \node at (0.866*0.5,0.25) [] {\footnotesize{$X_{(0,1,1)}$}};
        \node at (-0.866,-0.5) [] {\footnotesize{$X_{(2,0,0)}$}};
        \node at (0.866,-0.5) [] {\footnotesize{$X_{(0,2,0)}$}};
        \node at (0,-0.5) [] {\footnotesize{$X_{(1,1,0)}$}};
    \end{tikzpicture} \hspace{0.5cm}
    \begin{tikzpicture}[scale=1.8, baseline=0]
    \foreach \i in {1,...,5}
        \draw (0,2) -- (0,1.5);
        \draw (0,1.5) -- (-0.866*0.5,1.25);
        \draw (0,1.5) -- (0.866*0.5,1.25);
        \draw [fill=white] (0,2) circle [radius=0.1];
        \draw [fill=black] (0,1.5) circle [radius=0.1];
        \foreach \i in {0,1}
        {
        \draw (-0.866*0.5+0.866*\i,1.25) -- (-0.866*0.5+0.866*\i,0.75);
        \draw (-0.866*0.5+0.866*\i,0.75) -- (-0.866+0.866*\i,0.5);
        \draw (-0.866*0.5+0.866*\i,0.75) -- (0.866*\i,0.5);
        \draw [fill=white] (-0.866*0.5+0.866*\i,1.25) circle [radius=0.1];
        \draw [fill=black] (-0.866*0.5+0.866*\i,0.75) circle [radius=0.1];
        }
        \foreach \i in {0,1,2}
        {
        \draw (-0.866+0.866*\i,0.5) -- (-0.866+0.866*\i,0);
        \draw (-0.866+0.866*\i,0) -- (-0.866*1.5+0.866*\i,-0.25);
        \draw (-0.866+0.866*\i,0) -- (-0.866*0.5+0.866*\i,-0.25);
        \draw [fill=white] (-0.866+0.866*\i,0.5) circle [radius=0.1];
        \draw [fill=black] (-0.866+0.866*\i,0) circle [radius=0.1];
        }
        \foreach \i in {0,...,3}
        {
        \draw (-0.866*1.5+0.866*\i,-0.25) -- (-0.866*1.5+0.866*\i,-0.75);
        \draw (-0.866*1.5+0.866*\i,-0.75) -- (-0.866*2+0.866*\i,-1);
        \draw (-0.866*1.5+0.866*\i,-0.75) -- (-0.866+0.866*\i,-1);
        \draw [fill=white] (-0.866*1.5+0.866*\i,-0.25) circle [radius=0.1];
        \draw [fill=black] (-0.866*1.5+0.866*\i,-0.75) circle [radius=0.1];
        }
        \foreach \i in {0,...,4}
        {
        \draw [fill=white] (-2*0.866+0.866*\i,-1) circle [radius=0.1];
        }
        \node [blue] (00)  at (0.866*0.5,1.75)[ ] {\tiny{$\{2,3,4,5\}$}};
        \node [red] (10) at (0,1) [] {\tiny{$\{1,3,4,5\}$}};
        \node [blue] (11) at (0.866,1) [] {\tiny{$\{3,4,5\}$}};
        \node [blue] (22) at (0.866*1.5,0.25) [] {\tiny{$\{4,5\}$}};
        \node [blue] (33) at (0.866*2,-0.5) [] {\tiny{$\{5\}$}};
        \node [red] (20) at (-0.866*0.5,0.25) [] {\tiny{$\{1,2,4,5\}$}};
        \node [red] (21) at (0.866*0.5,0.25) [] {\tiny{$\{1,4,5\}$}};
        \node [red] (30) at (-0.866,-0.5) [] {\tiny{$\{1,2,3,5\}$}};
        \node [red] (32) at (0.866,-0.5) [] {\tiny{$\{1,5\}$}};
        \node [red] (31)at (0,-0.5) [] {\tiny{$\{1,2,5\}$}};
        \node [blue] (40) at (-0.866*1.5,-1.25) [] {\tiny{$\{1,2,3,4\}$}};
        \node [blue] (41) at (-0.866*0.5,-1.25) [] {\tiny{$\{1,2,3\}$}};
        \node [blue] (43) at (0.866*1.5,-1.25) [] {\tiny{$\{1\}$}};
        \node [blue] (42) at (0.866*0.5,-1.25) [] {\tiny{$\{1,2\}$}};
        \draw [red,->] (11) -- (10);
        \draw [red,->] (22) -- (21);
        \draw [red,->] (21) -- (20);
        \draw [red,->] (10) -- (00);
        \draw [red,->] (20) -- (10);
        \draw [red,->] (21) -- (11);
        \draw [red,->] (10) -- (21);
        \draw [red,->] (30) -- (20);
        \draw [red,->] (20) -- (31);
        \draw [red,->] (31) -- (21);
        \draw [red,->] (21) -- (32);
        \draw [red,->] (32) -- (22);
        \draw [red,->] (31) -- (30);
        \draw [red,->] (32) -- (31);
        \draw [red,->] (33) -- (32);
        \draw [red,->] (40) -- (30);
        \draw [red,->] (30) -- (41);
        \draw [red,->] (41) -- (31);
        \draw [red,->] (31) -- (42);
        \draw [red,->] (42) -- (32);
        \draw [red,->] (32) -- (43);
    \end{tikzpicture}
    \caption{Left: orientation on $\tilde{\bW}_5$ and the face variables. Right: the subsets $J_f$ associated with faces in $\bW_5$ and the extended quiver $\tilde{Q}_5$.}
    \label{fig: perfect orientation}
\end{figure}

Recall that the amalgamation morphism \cite{FGamalgamation} produces an upper triangulation matrix $x$ with entries being polynomials in the face variables $X_f$'s. On the other hand, the \emph{$(I,J)$th generalized minor} of $x$, denoted by $\Delta_I^J(x)$, is defined to be the determinant of the submatrix formed from the $I$-indexed rows and the $J$-indexed columns of $M$. By comparing the amalgamation morphism and the Postnikov's boundary measurement map, it is not hard to see that
\[
\Delta_I^J(x)=M_I^J.
\]
Thus, to compute the cluster DT transformation on $Q_n$, it suffices to know what ratios of generalized minors to take. By following \cite{Wengdb}, the generalized minors can be associated with faces of the ideal web $\bW_n$ as follows. We first equip the $k$th boundary face along the bottom edge of $S$ (counting from the left) with the set $\{1,2,\dots, n-k\}$, and equip the $k$th boundary face along the right edge of $S$ (counting from the top) with the set $\{k+1,k+2,\dots, n\}$. We then equip each bounded face $f$ corresponding to the triple $(a,b,c)$ with the subset
\begin{equation}\label{eq: def of J_f}
J_f=J_{(a,b,c)}=\{1,2,\dots, a+1, n-c,n-c+1,\dots, n\}.
\end{equation}
Note that it follows from the condition $a+b+c=n-3$ that $|J_{(a,b,c)}|=n-b-1$. We leave the boundary faces along the left edge of $S$ untouched. In the end, we extend the quiver $Q_n$ to a slightly bigger quiver $\tilde{Q}_n$ by adding a frozen vertex for each boundary face along the left and the bottom edge of $S$ (the ones associated with a subset of $\{1,2,\dots, n\}$, and then drawing arrows between these frozen vertices and the right most and bottom most quiver vertices of $Q_n$ so that they form counterclockwise $3$-cycles around the right most and bottom most black vertices in $\bW_n$. The right picture in Figure \ref{fig: perfect orientation} is a picture depicting the sets $J_f$'s and the extended quiver $\tilde{Q}_n$ for $n=5$. 

Now for each vertex $f$ in the extended quiver $\tilde{Q}_n$, we define a boundary measurement polynomial $M_f:=M_{\{1,2,\dots, |J_f|\}}^{J_f}$. We also denote the exchange matrix of $\tilde{Q}_n$ by $\tilde{\epsilon}$. 

\begin{thm} [{\cite[Theorem 1.1]{Wengdb}}] The cluster DT transformation on $Q_n$ acts on the cluster Poisson variables by
\begin{equation}\label{eq: DT in terms of M_f}
\DT(X_f)=\prod_g M_g^{\tilde{\epsilon}_{fg}}.
\end{equation}
\end{thm} 

We observe that for each boundary face $f$ in $\bW_n$, there is a unique family of pairwise non-intersecting paths in $C(I_f,J_f)$ and therefore $M_f$ is just a monomial. On the other hand, for each bounded face $f$ in $\bW_n$, we would like to relate the terms in the boundary measurement map to lozenge tilings of a hexagon.

\begin{defn} For a triple of non-negative integers $(a,b,c)$, we define $H_{(a,b,c)}$ to be a hexagon with side lengths $(a+1,b+1,c+1,a+1,b+1,c+1)$ and interior angles being all $120^\circ$. A \emph{lozenge tiling} of $H_{(a,b,c)}$ is a tiling of $H_{(a,b,c)}$ with the standard lozenge tile, which is a rhombus with all sides being unit length and interior angles being $(60^\circ, 120^\circ, 60^\circ,120^\circ)$.
\end{defn}

\begin{prop}\label{prop: bijection between terms in Mf and lozenge tilings} For each bounded face $f$ corresponding to a triple $(a,b,c)$, the terms in $M_f$ are in bijection with lozenge tilings of $H_{(a,b,c)}$.
\end{prop}
\begin{proof} Based on the way $\tilde{\bW}_n$ is oriented, there are no choices for the paths that end at the first $a+1$ indices $1,2,\dots, a+1$ of $J_{(a,b,c)}$. For the remaining $c+1$ many indices, we can denote them as $b+k+1$ for $a+2\leq k\leq a+c+2$. Then the path ending at the sink $b+k+1$ in $J_{(a,b,c)}$ must come in from the source $k$. Due to the restrictions imposed by this pairing of the sources and sinks, we see that the only freedom in choosing different families of pairwise disjoint paths lies inside a confined region of the ideal web $\bW_n$. To be precise, let us define the \emph{hexagonal span} of the vertex $(a,b,c)$ of the quiver $Q_n$ to be the full subquiver $Q_{(a,b,c)}$ defined by the quiver vertices inside the closed convex hull of the points 
\[
(a,0,b+c),\quad (0,a,b+c), \quad (0,a+b,c), \quad (c,0,a+b), \quad (a+c,b,0), \quad (a+c,0,b). 
\]
Then different families of pairwise disjoint paths from $I_{(a,b,c)}$ to $J_{(a,b,c)}$ only differ within the edges adjacent to the faces corresponding to vertices inside $Q_{(a,b,c)}$. Figure \ref{fig: hexagonal region} highlights the hexagonal span $Q_{(2,1,1)}$ in $Q_7$ and the corresponding region in $\bW_7$.
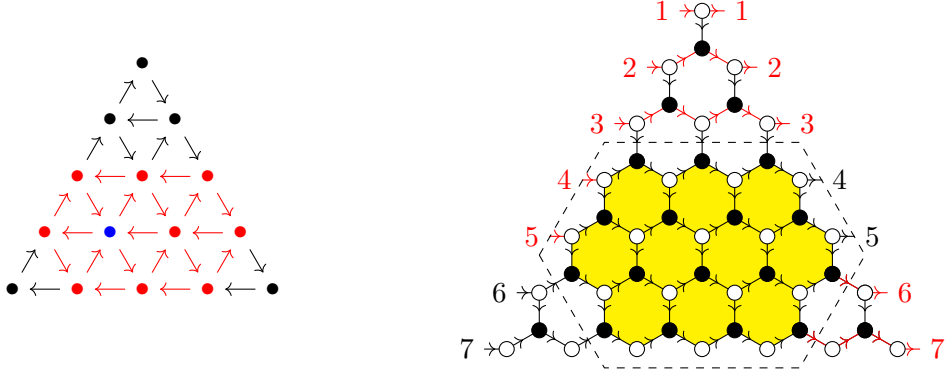
\begin{figure}[H]
    \centering
    \begin{tikzpicture}[baseline=0]
        \node (0) at (-150:2) [] {$\bullet$};
        \node [red] (1) at (-0.866,-1) [] {$\bullet$};
        \node [red] (2) at (-90:1) [] {$\bullet$};
        \node [red] (3) at (0.866,-1) [] {$\bullet$};
        \node (4) at (-30:2) [] {$\bullet$};
        \node [red] (5) at (-0.866*1.5,-0.25) [] {$\bullet$};
        \node [blue] (6) at (-150:0.5) [] {$\bullet$};
        \node [red] (7) at (-30:0.5) [] {$\bullet$};
        \node [red] (8) at (0.866*1.5,-0.25) [] {$\bullet$};
        \node [red] (9) at (-0.866,0.5) [] {$\bullet$};
        \node [red] (10) at (0,0.5) [] {$\bullet$};
        \node [red] (11)  at (0.866,0.5) [] {$\bullet$};
        \node (12) at (-0.866*0.5,1.25) [] {$\bullet$};
        \node (13) at (0.866*0.5,1.25) [] {$\bullet$};
        \node (14) at (0,2) [] {$\bullet$};
        \draw [->] (1) -- (0);
        \draw [->] (0) -- (5);
        \draw [red,->] (2) -- (1);
        \draw [red,->] (3) -- (2);
        \draw [->] (4) -- (3);
        \draw [red,->] (5) --(1);
        \draw [red,->] (1) -- (6);
        \draw [red,->] (6) -- (2);
        \draw [red,->] (2) --(7);
        \draw [red,->] (7) -- (3);
        \draw [red,->] (3) -- (8);
        \draw[->] (8) -- (4);
        \draw [red,->] (6) -- (5);
        \draw [red,->] (7) -- (6);
        \draw [red,->] (8) -- (7);
        \draw [red,->] (5) -- (9);
        \draw [red,->] (9) -- (6);
        \draw [red,->] (6) -- (10);
        \draw [red,->] (10) -- (7);
        \draw [red,->] (7) -- (11);
        \draw [red,->] (11) -- (8);
        \draw [red,->] (10) -- (9);
        \draw [red,->] (11) -- (10);
        \draw [->] (9) -- (12);
        \draw [->] (12) -- (10);
        \draw [->] (10) -- (13);
        \draw [->] (13) -- (11);
        \draw [->] (13) -- (12);
        \draw [->] (12) -- (14);
        \draw [->] (14) -- (13);
    \end{tikzpicture}\hspace{2cm}
    \begin{tikzpicture}[baseline=-20]
    \path [fill=yellow] (0.866-0.866*0.5*5, 2.75-0.75*5) -- (0.866-0.866*0.5*5+0.866*0.5, 2.75-0.75*5+0.25) -- (0.866-0.866*0.5*5+0.866*0.5, 2.75-0.75*5+0.75) -- (0.866-0.866*0.5*5+0.866, 2.75-0.75*5+1) -- (0.866-0.866*0.5*5+0.866*1.5, 2.75-0.75*5+0.75)--(0.866-0.866*0.5*5+0.866*2, 2.75-0.75*5+1) -- (0.866-0.866*0.5*5+0.866*2.5, 2.75-0.75*5+0.75) -- (0.866-0.866*0.5*5+0.866*3, 2.75-0.75*5+1) -- (0.866-0.866*0.5*5+0.866*3.5, 2.75-0.75*5+0.75) -- (0.866-0.866*0.5*5+0.866*3.5, 2.75-0.75*5+0.25)-- (0.866-0.866*0.5*5+0.866*4, 2.75-0.75*5+0) -- (0.866-0.866*0.5*5+0.866*4, 2.75-0.75*5-0.5) -- (0.866-0.866*0.5*5+0.866*3.5, 2.75-0.75*5-0.75) -- (0.866-0.866*0.5*5+0.866*3.5, 2.75-0.75*5-1.25) -- (0.866-0.866*0.5*5+0.866*3, 2.75-0.75*5-1.5)-- (0.866-0.866*0.5*5+0.866*2.5, 2.75-0.75*5-1.25) -- (0.866-0.866*0.5*5+0.866*2, 2.75-0.75*5-1.5) -- (0.866-0.866*0.5*5+0.866*1.5, 2.75-0.75*5-1.25) -- (0.866-0.866*0.5*5+0.866*1, 2.75-0.75*5-1.5)-- (0.866-0.866*0.5*5+0.866*0.5, 2.75-0.75*5-1.25)-- (0.866-0.866*0.5*5+0.866*0.5, 2.75-0.75*5-0.75) --(0.866-0.866*0.5*5, 2.75-0.75*5-0.5) -- cycle; 
    \foreach \i in {1,2,3}
    {
    \draw[red, decoration={markings, mark=at position 0.5 with {\arrow{>}}}, postaction={decorate}] (-0.3-0.866*\i*0.5+0.866,2.75-0.75*\i) node [left] {$\i$} -- (-0.866*\i*0.5+0.866,2.75-0.75*\i);
    \draw[red, decoration={markings, mark=at position 0.7 with {\arrow{>}}}, postaction={decorate}] (0.866*\i*0.5,2.75-0.75*\i) -- (0.3+0.866*\i*0.5,2.75-0.75*\i) node[right] {$\i$};
    }
     \foreach \i in {4,5}
    {
    \draw[red, decoration={markings, mark=at position 0.5 with {\arrow{>}}}, postaction={decorate}] (-0.3-0.866*\i*0.5+0.866,2.75-0.75*\i) node [left] {$\i$} -- (-0.866*\i*0.5+0.866,2.75-0.75*\i);
    \draw[decoration={markings, mark=at position 0.7 with {\arrow{>}}}, postaction={decorate}] (0.866*\i*0.5,2.75-0.75*\i) -- (0.3+0.866*\i*0.5,2.75-0.75*\i) node[right] {$\i$};
    }
     \foreach \i in {6,7}
    {
    \draw[decoration={markings, mark=at position 0.5 with {\arrow{>}}}, postaction={decorate}] (-0.3-0.866*\i*0.5+0.866,2.75-0.75*\i) node [left] {$\i$} -- (-0.866*\i*0.5+0.866,2.75-0.75*\i);
    \draw[red, decoration={markings, mark=at position 0.7 with {\arrow{>}}}, postaction={decorate}] (0.866*\i*0.5,2.75-0.75*\i) -- (0.3+0.866*\i*0.5,2.75-0.75*\i) node[right] {$\i$};
    }
    \foreach \i in {1,2}
    {
     \foreach \j in {1,...,\i}
     {
     \draw[decoration={markings, mark=at position 0.5 with {\arrow{>}}}, postaction={decorate}] (0.866*\j-0.866*0.5*\i, 2.75-0.75*\i) -- (0.866*\j-0.866*0.5*\i,2.25-0.75*\i);
     \draw[red,decoration={markings, mark=at position 0.5 with {\arrow{>}}}, postaction={decorate}] (0.866*\j-0.866*0.5*\i-0.866*0.5, 2-0.75*\i) -- (0.866*\j-0.866*0.5*\i,2.25-0.75*\i);
     \draw[red,decoration={markings, mark=at position 0.5 with {\arrow{>}}}, postaction={decorate}] (0.866*\j-0.866*0.5*\i,2.25-0.75*\i) -- (0.866*\j-0.866*0.5*\i+0.866*0.5, 2-0.75*\i);
     \draw [fill=white] (0.866*\j-0.866*0.5*\i, 2.75-0.75*\i) circle [radius=0.1];
     \draw [fill=black] (0.866*\j-0.866*0.5*\i,2.25-0.75*\i) circle [radius=0.1];
     }
    } 
    \foreach \i in {3,...,6}
    {
     \foreach \j in {1,...,\i}
     {
     \draw[decoration={markings, mark=at position 0.5 with {\arrow{>}}}, postaction={decorate}] (0.866*\j-0.866*0.5*\i, 2.75-0.75*\i) -- (0.866*\j-0.866*0.5*\i,2.25-0.75*\i);
     \draw[decoration={markings, mark=at position 0.5 with {\arrow{>}}}, postaction={decorate}] (0.866*\j-0.866*0.5*\i-0.866*0.5, 2-0.75*\i) -- (0.866*\j-0.866*0.5*\i,2.25-0.75*\i);
     \draw[decoration={markings, mark=at position 0.5 with {\arrow{>}}}, postaction={decorate}] (0.866*\j-0.866*0.5*\i,2.25-0.75*\i) -- (0.866*\j-0.866*0.5*\i+0.866*0.5, 2-0.75*\i);
     \draw [fill=white] (0.866*\j-0.866*0.5*\i, 2.75-0.75*\i) circle [radius=0.1];
     \draw [fill=black] (0.866*\j-0.866*0.5*\i,2.25-0.75*\i) circle [radius=0.1];
     }
    } 
    \draw[red, decoration={markings, mark=at position 0.5 with {\arrow{>}}}, postaction={decorate}] (0.866*5-0.866*0.5*5,2.25-0.75*5) -- (0.866*5-0.866*0.5*5+0.866*0.5, 2-0.75*5);
    \draw[red, decoration={markings, mark=at position 0.5 with {\arrow{>}}}, postaction={decorate}] (0.866*5-0.866*0.5*6,2.25-0.75*6) -- (0.866*5-0.866*0.5*6+0.866*0.5, 2-0.75*6);
        \draw[red, decoration={markings, mark=at position 0.5 with {\arrow{>}}}, postaction={decorate}] (0.866*6-0.866*0.5*6,2.25-0.75*6) -- (0.866*6-0.866*0.5*6+0.866*0.5, 2-0.75*6);
         \draw[red,decoration={markings, mark=at position 0.5 with {\arrow{>}}}, postaction={decorate}] (0.866*6-0.866*0.5*6-0.866*0.5, 2-0.75*6) -- (0.866*6-0.866*0.5*6,2.25-0.75*6);
    \draw [fill=white] (0.866*6-0.866*0.5*6, 2.75-0.75*6) circle [radius=0.1];
     \draw [fill=black] (0.866*5-0.866*0.5*5,2.25-0.75*5) circle [radius=0.1];
     \draw [fill=black] (0.866*5-0.866*0.5*6,2.25-0.75*6) circle [radius=0.1];
     \draw [fill=black] (0.866*6-0.866*0.5*6,2.25-0.75*6) circle [radius=0.1];
    \foreach \j in {1,...,7}
    {
    \draw [fill=white] (0.866*\j-0.866*0.5*7, 2.75-0.75*7) circle [radius=0.1];
    }
    \draw [dashed]  (0.866*0.5-0.866*0.5*5, 2.5-0.75*5) -- (0.866*0.5-0.866*0.5*5+0.866, 2.5-0.75*5+1.5) --  (0.866*0.5-0.866*0.5*5+0.866*4, 2.5-0.75*5+1.5) -- (0.866*0.5-0.866*0.5*5+0.866*5, 2.5-0.75*5) -- (0.866*0.5-0.866*0.5*5+0.866*4, 2.5-0.75*5-1.5) -- (0.866*0.5-0.866*0.5*5+0.866, 2.5-0.75*5-1.5) -- cycle ;
    \end{tikzpicture}
    \caption{The hexagonal span $Q_{(2,1,1)}$ and its corresponding region where families of pairwise disjoint paths differ. Note that subset associated with $(2,1,1)$ is $J_{(2,1,1)} = \{1,2,3,6,7\}$ (see Equation \eqref{eq: def of J_f}), and the source set is $\{1,2,3,4,5\}$. In particular, we color the part shared by any families of pairwise disjoint paths in red in the picture on the right.}
    \label{fig: hexagonal region}
\end{figure}
\begin{figure}[H]
    \centering
    \begin{tikzpicture}
        \foreach \i in {1,2,3}
        {
         \draw [lightgray, decoration={markings, mark=at position 0.5 with {\arrow{>}}}, postaction={decorate}] (0.866*\i, 2.25) -- (0.866*\i+0.866*0.5,2.5);
        \draw [lightgray, decoration={markings, mark=at position 0.5 with {\arrow{>}}}, postaction={decorate}] (0.866*\i+0.866*0.5, 2.5) -- (0.866*\i+0.866,2.25);
        \draw [lightgray, fill=lightgray] (0.866*\i+0.866*0.5,2.5) circle [radius=0.1];
        \draw [lightgray, decoration={markings, mark=at position 0.5 with {\arrow{>}}}, postaction={decorate}] (0.866*\i, 0.25) -- (0.866*\i+0.866*0.5,0);
        \draw [lightgray, decoration={markings, mark=at position 0.5 with {\arrow{>}}}, postaction={decorate}] (0.866*\i+0.866*0.5, 0) -- (0.866*\i+0.866,0.25);
        \draw [lightgray, fill=white] (0.866*\i+0.866*0.5,0) circle [radius=0.1];
        }
        \foreach \i in {1,...,4}
        {
        \draw [lightgray, decoration={markings, mark=at position 0.5 with {\arrow{>}}}, postaction={decorate}] (0.866*\i, 2.25) -- (0.866*\i,1.75);
        \draw [lightgray, decoration={markings, mark=at position 0.5 with {\arrow{>}}}, postaction={decorate}] (0.866*\i, 0.75) -- (0.866*\i,0.25);
        \draw [lightgray, decoration={markings, mark=at position 0.5 with {\arrow{>}}}, postaction={decorate}] (0.866*\i, 1.75) -- (0.866*\i+0.866*0.5,1.5);
        \draw [lightgray, decoration={markings, mark=at position 0.5 with {\arrow{>}}}, postaction={decorate}] (0.866*\i-0.866*0.5,1.5) --(0.866*\i, 1.75);
        \draw [lightgray, decoration={markings, mark=at position 0.5 with {\arrow{>}}}, postaction={decorate}] (0.866*\i, 0.75) -- (0.866*\i+0.866*0.5,1);
        \draw [lightgray, decoration={markings, mark=at position 0.5 with {\arrow{>}}}, postaction={decorate}] (0.866*\i-0.866*0.5,1) --(0.866*\i, 0.75);
        \draw [lightgray, fill=white] (0.866*\i,2.25) circle [radius=0.1];
        \draw [lightgray, fill=lightgray] (0.866*\i,1.75) circle [radius=0.1];
        \draw [lightgray, fill=white] (0.866*\i,0.75) circle [radius=0.1];
        \draw [lightgray, fill=lightgray] (0.866*\i,0.25) circle [radius=0.1];
        }
        \foreach \i in {1,...,5}
        {
        \draw [lightgray, decoration={markings, mark=at position 0.5 with {\arrow{>}}}, postaction={decorate}] (0.866*\i-0.866*0.5, 1.5) -- (0.866*\i-0.866*0.5,1);
        \draw [lightgray, fill=white] (0.866*\i-0.866*0.5,1.5) circle [radius=0.1];
        \draw [lightgray, fill=lightgray] (0.866*\i-0.866*0.5,1) circle [radius=0.1];
        }
        \draw [dashed] (0.866,-0.25) -- (0,1.25) -- (0.866,2.75) -- (4*0.866,2.75) -- (5*0.866,1.25) -- (4*0.866,-0.25) -- cycle;
        \draw [line width=5, pink] (0.866,2.25) -- (0.866*1.5,2.5);
        \draw [line width=5, pink] (0.866*2,2.25) -- (0.866*2,1.75);
        \draw [line width=5, pink] (0.866*2.5,1.5) -- (0.866*3,1.75);
        \draw [line width=5, pink] (0.866*3.5,1.5) -- (0.866*3.5,1);
        \draw [line width=5, pink] (0.866*4,0.75) -- (0.866*4.5,1);
        \draw [line width=5, pink] (0.866*0.5,1.5) -- (0.866*0.5,1);
        \draw [line width=5, pink] (0.866*1,0.75)--(0.866*1.5,1);
        \draw [line width=5, pink] (0.866*2,0.75) --(0.866*2.5,1);
        \draw [line width=5, pink] (0.866*3,0.75) -- (0.866*3,0.25);
        \draw [line width=5, pink] (0.866*3.5,0) -- (0.866*4,0.25);
        \draw [->, red] (0.866*0.5, 2.5) -- (0.866,2.25) -- (0.866*1.5,2.5) -- (0.866*2,2.25) -- (0.866*2,1.75) -- (0.866*2.5,1.5) -- (0.866*3,1.75) -- (0.866*3.5,1.5) -- (0.866*3.5,1) -- (0.866*4,0.75) -- (0.866*4.5,1) -- (0.866*5,0.75);
        \draw [red, ->] (0,1.75) -- (0.866*0.5,1.5) -- (0.866*0.5,1) -- (0.866*1,0.75)--(0.866*1.5,1) -- (0.866*2,0.75) --(0.866*2.5,1) -- (0.866*3,0.75) -- (0.866*3,0.25) -- (0.866*3.5,0) -- (0.866*4,0.25) -- (0.866*4.5,0);
        \draw [dashed] (0.866*0.5,2) -- (0.866,1.25) -- (0.866*0.5,0.5);
        \draw [dashed] (0.866*0.5,2) -- (0.866*1.5,2) -- (0.866*2,2.75);
        \draw [dashed] (0.866,1.25) -- (0.866*2,1.25) -- (0.866*1.5,0.5) -- (0.866*0.5,0.5);
        \draw [dashed] (0.866*2,2.75) -- (0.866*2.5,2) -- (0.866*2,1.25) -- (0.866*1.5,2);
        \draw [dashed] (0.866*2,1.25) -- (0.866*3,1.25) -- (0.866*3.5,2) -- (0.866*2.5,2);
        \draw [dashed] (0.866*1.5,0.5) -- (0.866*2.5,0.5) -- (0.866*3,1.25);
        \draw [dashed] (0.866*2.5,0.5) -- (0.866*3,-0.25) -- (0.866*3.5,0.5) -- (0.866*3,1.25);
        \draw [dashed] (0.866*3.5,0.5) -- (0.866*4,1.25) -- (0.866*3.5,2);
        \draw [dashed] (0.866*2,-0.25) -- (0.866*1.5,0.5);
        \draw [dashed] (0.866*3,2.75) -- (0.866*3.5,2) -- (0.866*4.5,2);
        \draw [dashed] (0.866*4,1.25) -- (0.866*5,1.25);
        \draw [dashed] (0.866*3.5,0.5) -- (0.866*4.5,0.5);
    \end{tikzpicture} \hspace{2cm}
    \begin{tikzpicture}
        \draw [fill=lightgray] (0.866*0.5,2) -- (0.866,1.25) -- (0.866*0.5,0.5) -- (0,1.25) -- cycle;
        \draw (0.866*0.5,2) -- (0.866*1.5,2) -- (0.866*2,2.75) -- (0.866,2.75) -- cycle;
        \draw [fill= gray] (0.866,1.25) -- (0.866*2,1.25) -- (0.866*1.5,2) -- (0.866*0.5,2) -- cycle;
        \draw [fill=lightgray] (0.866*2,2.75) -- (0.866*2.5,2) -- (0.866*2,1.25) -- (0.866*1.5,2) -- cycle;
        \draw [fill=gray] (0.866*2,2.75) -- (0.866*4,2.75) -- (0.866*5,1.25) -- (0.866*4,1.25) -- (0.866*3.5,2) -- (0.866*2.5,2) -- cycle;
        \draw [fill=lightgray] (0.866*4,1.25) -- (0.866*3,-0.25) -- (0.866*2.5,0.5) -- (0.866*3.5,2) -- cycle;
        \draw [fill=gray] (0.866*0.5,0.5) -- (0.866,-0.25) -- (0.866*3,-0.25) -- (0.866*2.5,0.5) -- cycle;
        \draw (0.866*5,1.25) -- (0.866*4,-0.25) -- (0.866*3,-0.25);
        \draw (0.866*2,-0.25) -- (0.866*1.5,0.5) -- (0.866*2,1.25) -- (0.866*3,1.25) -- (0.866*3.5,0.5) -- (0.866*4.5,0.5);
        \draw (0.866*4.5,2) -- (0.866*3.5,2) -- (0.866*3,2.75);
    \end{tikzpicture}
    \caption{Left: example of a lozenge tiling corresponding to a family of pairwise disjoint paths. Right: the corresponding 3D Young diagram, which consists of five cubes sitting inside a box of size $3\times 2\times 2$.}
    \label{fig: lozenge tiling}
\end{figure}
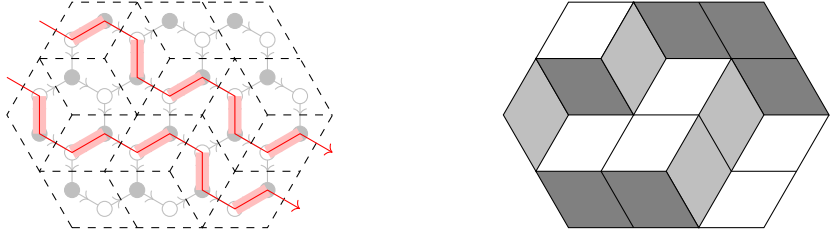
This region of freedom in $\bW_n$ is very close to being the hexagon $H_{(a,b,c)}$ itself. All we need to do is to draw a few more hexagons congruent to the bounded faces of $\bW_n$ surrounding this region of freedom, and then connect the centers of the adjacent hexagons with straight lines. In Figure \ref{fig: hexagonal region} we draw $H_{(2,1,1)}$ in dashed lines. Note that the side length of $H_{(a,b,c)}$ is precisely $(a+1,b+1,c+1,a+1,b+1,c+1)$ scaled by a factor of $\sqrt{2}$, which we may ignore for simplicity.

Next, we observe that for each family $\Gamma$ of pairwise disjoint paths, exactly $c+1$ many of them go through $H_{(a,b,c)}$: they all enter from the upper left edge of $H_{(a,b,c)}$ and exit through the lower right edge of $H_{(a,b,c)}$. To obtain a lozenge tiling of $H_{(a,b,c)}$ from $\Gamma$, for each path $\gamma$ in $\Gamma$ that goes through $H_{(a,b,c)}$, we take the first edge that is completely contained inside $H_{(a,b,c)}$, and then select every other edge along $\gamma$. This gives a collection of edges inside $H_{(a,b,c)}$, and place a lozenge tile with each of these edges in the middle as the line segment connecting the two equilateral triangles inside the lozenge tile. These lozenge tiles in turn determine a unique lozenge tiling of the whole hexagon $H_{(a,b,c)}$. The left picture in Figure \ref{fig: lozenge tiling} gives an example of one such lozenge tiling corresponding to a family of pairwise disjoint paths. Conversely, given any lozenge tiling of $H_{(a,b,c)}$ we can recover the paths $\gamma$ by working our way through the tilings from the upper left edge to the lower right edge. This proves that there is a bijection between terms in the boundary measurement polynomial $M_{f_{(a,b,c)}}$ and lozenge tilings of $H_{(a,b,c)}$.
\end{proof}

It is well-known that lozenge tilings of the hexagon $H_{(a,b,c)}$ are in turn in bijection with \emph{3D Young diagrams} (also known as \emph{plane partitions} that fit inside a box of size $(a+1)\times (b+1)\times (c+1)$. The right picture in Figure \ref{fig: lozenge tiling} shows the 3D Young diagram corresponding to the lozenge tiling in the left picture of the same Figure. Note that the set of 3D Young diagrams 
 inside a box admits a partial order by inclusion, which has a unique minimal element, namely the empty diagram, and a unique maximal element, namely the diagram that fills the whole box. By considering the family of pairwise disjoint paths associated with the empty Young diagram, we see that the term in $M_f$ corresponding to the empty Yound diagram dominates the smallest amount of faces and divides all other terms in $M_f$. Thus, we can factor out this term $N_f$ and write
 \[
 M_f=N_f\Phi_f
 \]
 where $\Phi_f$ is a polynomial in $X_g$'s with a constant term $1$. By convention, we set $N_f=M_f$ for boundary faces, since $M_f$ is itself a monomial. Now Equation \eqref{eq: DT in terms of M_f} can be written as
 \begin{equation}\label{eq: DT of X variable}
\DT(X_f)=\left(\prod_gN_g^{\tilde{\epsilon}_{fg}}\right)\left(\prod_g\Phi_g^{\epsilon_{fg}}\right).
\end{equation}

We claim that $\Phi_f$ can be identified with an ideal function of a simply-labeled pointed poset. In fact, we already have a poset associated with the face $f_{(a,b,c)}$, namely the set of 3D Young diagrams inside a box of size $(a+1)\times (b+1)\times (c+1)$. But we can obtain the same poset by performing a ``3D lift'' of the hexagon span $Q_{(a,b,c)}$. The \emph{3D lift} of the full subquiver $Q_{(a,b,c)}$ is an acyclic quiver with vertex set
\[
\{(x,y,z)\mid 0\leq x\leq c, 0\leq y\leq a, 0\leq z\leq b\}.
\]
Note that this vertex set admits a surjection onto the vertex set of $Q_{(a,b,c)}$ by
\[
\pi:(x,y,z)\mapsto (a,b,c)+x(1,0,-1)+y(-1,1,0)+z(0,-1,1).
\]
Now for every pair of vertices $i$ and $j$ of distance $1$ in the 3D lift, we define $\epsilon_{ij}:=\epsilon^{Q_n}_{\pi(i)\pi(j)}$. Since this 3D lift is an acyclic quiver, it is naturally a Hasse diagram and hence in turn defines a poset $P(a,b,c)$. Based on this construction, it is not hard to see the following fact.

\begin{lem}\label{lem: isomorphic posets} $P(a,b,c)$ is isomorphic to the post of 3D Young diagrams inside a box of size $(a+1)\times (b+1)\times (c+1)$.
\end{lem}

We label each element $i$ in $P(a,b,c)$ by $L(a,b,c) :X_{\pi(i)}$. Below is an example of such labeled posets.

\begin{figure}[H]
    \centering
    \begin{tikzpicture}[baseline=0,scale=0.8]
    \node [blue] (0) at (0,0) [] {\footnotesize{$X_{(2,1,1)}$}};
    \node (1) at (-1,-1) [] {\footnotesize{$X_{(3,1,0)}$}};
    \node (2) at (-1,1.5) [] {\footnotesize{$X_{(3,0,1)}$}};
    \node (3) at (0,2.5) [] {\footnotesize{$X_{(2,0,2)}$}};
    \node (4) at (2.5,0) [] {\footnotesize{$X_{(1,2,1)}$}};
    \node (5) at (1.5,-1) [] {\footnotesize{$X_{(2,2,0)}$}};
    \node [blue] (6) at (1.5,1.5) [] {\footnotesize{$X_{(2,1,1)}$}};
    \node (7) at (2.5,2.5) [] {\footnotesize{$X_{(1,1,2)}$}};
    \node (8) at (5,0) [] {\footnotesize{$X_{(0,3,1)}$}};
    \node (9) at (4,-1) [] {\footnotesize{$X_{(1,3,0)}$}};
    \node (10) at (4,1.5) [] {\footnotesize{$X_{(1,2,1)}$}};
    \node (11) at (5,2.5) [] {\footnotesize{$X_{(0,2,2)}$}};
    \draw [->] (1) -- (0);
    \draw [->] (2) -- (1);
    \draw [->] (2) -- (3);
    \draw [->] (3) -- (0);
    \draw [->] (4) -- (0);
    \draw [->] (5) -- (1);
    \draw [->] (6) -- (2);
    \draw [->] (7) -- (3);
    \draw [->] (5) -- (4);
    \draw [->] (6) -- (5);
    \draw [->] (6) -- (7);
    \draw [->] (7) -- (4);
    \draw [->] (8) -- (4);
    \draw [->] (9) -- (5);
    \draw [->] (10) -- (6);
    \draw [->] (11) -- (7);
    \draw [->] (9) -- (8);
    \draw [->] (10) -- (9);
    \draw [->] (10) -- (11);
    \draw [->] (11) -- (8);
    \end{tikzpicture}
    \caption{The 3D lift of the hexagonal span $Q_{(2,1,1)}$ in $Q_7$ together with the simple labeling.}
    \label{fig: hexagonal span}
\end{figure}
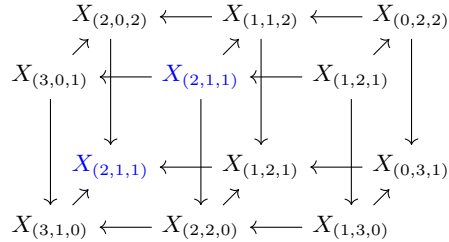

By combining Proposition \ref{prop: bijection between terms in Mf and lozenge tilings} and Lemma \ref{lem: isomorphic posets}, and comparing them with the definition of the boundary measurement polynomial, we can deduce the following conclusion.

\begin{prop} Let $(a,b,c)$ be the triple associated with a bounded face $f$ in $\bW_n$. Then $\Phi_f$ is equal to the ideal function $F_{(P(a,b,c),L(a,b,c))}$.
\end{prop}
\begin{proof} The partial order on 3d Young diagrams can be interpreted as the following move on the lozenge tilings: $\begin{tikzpicture}[baseline=-3]
    \foreach \i in {0,...,5}
    {
    \draw (60+\i*60:0.5) -- (120+\i*60:0.5);
    }
    \foreach \i in {0,...,2}
    {
    \draw (120*\i:0.5) -- (0,0);
    }
\end{tikzpicture}< \begin{tikzpicture}[baseline=-3]
    \foreach \i in {0,...,5}
    {
    \draw (60+\i*60:0.5) -- (120+\i*60:0.5);
    }
    \foreach \i in {0,...,2}
    {
    \draw (60+120*\i:0.5) -- (0,0);
    }
\end{tikzpicture}$. This move exactly increases the number of dominated faces (counted with multiplicity) by $1$, and this newly added face is precisely the index of the labeling of the newly added cube in the 3D Young diagram.
\end{proof}

Let us now relate $\Phi_f$ with $F$-polynomials of the cluster DT transformation. In \cite{Wengdb}, Weng proves that $\prod_gN_g^{\tilde{\epsilon}_{fg}}=X_f^{-1}$, and therefore Equation \eqref{eq: DT of X variable} can be simplified to
\begin{equation}\label{eq3} 
\DT(X_f)=X_f^{-1}\prod_g \Phi_g^{\epsilon_{fg}}.
\end{equation}
We observe that this equation resembles the second equation in the separation formula for DT transformations \eqref{eq: DT F polynomial}. This motivates the formulation of the following theorem.

\begin{thm}\label{thm: F polynomials for DT of triple of flags} For any bounded face $f$ of $\bW_n$, $F_f=\Phi_f$.
\end{thm}

Since the exchange matrix $\epsilon$ of $Q_n$ is not necessarily invertible, we cannot directly conclude that $F_f=\Phi_f$. 

Instead, we need to make use of a concrete maximal green sequence on $Q_n$. It is known (e.g., \cite{GS2}) that a maximal green sequence on $Q_n$ can be constructed by the following iterative procedure.
\begin{itemize}
    \item If $n=3$, then a single mutation at the unique mutable vertex is a maximal green sequence. 
    \item For $n>3$, the quiver $Q_{n-1}$ embeds as a full subquiver inside $Q_n$ at the lower left hand corner. For each $0\leq a\leq n-3$, let $\bk_a$ be the mutation sequence $((n-a-3,a,0),(n-a-3,a-1,1), (n-a-3,a-2,2),\dots, (n-a-3,0,a))$. Then the composition $(\bk_0,\bk_1,\dots, \bk_{n-3})$ mutates only at green vertices and turns the right most quiver vertex in each row red. We then postcompose this mutation sequence by the maximal green sequence of $Q_{n-1}$ inductively. 
\end{itemize}
After applying this maximal green sequence, the resulting quiver is actually the mirror image of $Q_n$ after flipping over the vertical line of symmetry. Thus, the cluster DT transformation is the composition of this maximal green sequence together with the quiver isomorphism of flipping the quiver horizontally.

\begin{lem}\label{lem: right most vertices} Theorem \ref{thm: F polynomials for DT of triple of flags} is true for the faces with triples $(a,0,n-a-3)$.
\end{lem}
\begin{proof} Due to the horizontal flip, the $F$-polynomials of the cluster DT transformation at vertices $(a,0,n-a-3)$ are the $F$-polynomials of the maximal green sequence at vertices $(0,a,n-a-3)$, i.e., the right most vertices of the rows in $Q_n$. We then observe that the maximal green sequence only mutates at the right most vertex of each row once, which happens during the $\bk_{n-3}$ subsequence inside the subsequence $(\bk_0,\bk_1,\dots, \bk_{n-3})$. This implies that the $F$-polynomials of interest in this lemma are already settled after the mutation sequence $(\bk_0,\bk_1,\dots, \bk_{n-3})$.

Let us now do an induction on $n$. For the base case $n=3$, there is one vertex and $F_{(0,0,0)}=1+X_{(0,0,0)}$, which is equal to the ideal function $\Phi_{(0,0,0)}=[X_{(0,0,0)}]$. Now inductively suppose the Lemma is true for $n-1$. By viewing $Q_{n-1}$ as a full subquiver sitting at the lower left hand corner of $Q_n$, we see inductively that the $F$-polynomial $F'_{(1,b,c)}$ after the mutation sequence $(\bk_0,\bk_1,\dots, \bk_{n-4})$ is precisely the $F$-polynomial $F_{(b,0,c)}^{Q_{n-1}}$ of DT on $Q_{n-1}$, and is equal to the ideal function $\Phi_{(b,0,c)}^{Q_{n-1}}$ (we use a superscript here to indicate that these are the $F$-polynomials associated with $Q_{n-1}$ rather than $Q_n$). On the other hand, by a direct computation one also finds that the $C$-vector of the $(0,a,n-a-3)$ after the mutation sequence $(\bk_0,\bk_1,\dots, \bk_{n-4})$ has $1$'s in every entry in the same row as $(0,a,n-a-3)$ and z$0$'s everywhere else \cite[Proposition 4.3]{SWflag}.

Let us now consider what happens at $(0,a,n-a-3)$ in the next mutation subsequence $\bk_{n-3}$. When $a=n-3$, then the quiver right before mutating at $(0,n-3,0)$ looks like the left picture in Figure \ref{fig: local picture during mutation}. Note that by the inductive hypothesis,
\[
F'_{(1,n-4,0)}=\Phi_{(n-4,0,0)}^{Q_{n-1}}=[X_{(n-3,0,0)}\leftarrow X_{(n-2,1,0)}\leftarrow \cdots \leftarrow X_{(1,n-4,0)}].
\]
Thus, by applying the $F$-polynomial mutation formula \eqref{eq: F mutation}, we get
\begin{align*}
F_{(n-3,0,0)}=F'_{(0,n-3,0)}=&[X_{(n-3,0,0)}\leftarrow X_{(n-2,1,0)}\leftarrow \cdots \leftarrow X_{(1,n-4,0)}]+\prod_{k=0}^{n-3}X_{(n-k-3,k,0)}\\
=&[X_{(n-3,0,0)}\leftarrow X_{(n-2,1,0)}\leftarrow \cdots \leftarrow X_{(1,n-4,0)}\leftarrow X_{(0,n-3,0)}]\\
=&\Phi_{(n-3,0,0)}.
\end{align*}

\begin{figure}[H]
    \centering
    \begin{tikzpicture}[baseline=30]
        \node (0) [red] at (0,0) [] {$(1,n-4,0)$};
        \node (1) [teal] at (3,0) [] {$(0,n-3,0)$};
        \draw [->] (0) -- (1);
    \end{tikzpicture}  
    \begin{tikzpicture}[baseline=0]
        \node (0) [red] at (0,0) [] {$(1,a-1,n-a-3)$};
        \node (1) [teal] at (4,0) [] {$(0,a,n-a-3)$};
        \node (2) [red] at (1.5,-1) [] {$(1,a,n-a-4)$};
        \draw [->] (0) -- (1);
        \draw [->] (1) -- (2);
    \end{tikzpicture} 
    \begin{tikzpicture}[baseline=0]
     \node (0) [red] at (-1.5,-1) [] {$(1,0,n-4)$};
     \node (1) [teal] at (0,0) [] {$(0,0,n-3)$};
     \draw [->] (1) -- (0);
    \end{tikzpicture}
    \caption{Local picture near the vertex $(0,a,n-a-3)$ right before the mutation happens.}
    \label{fig: local picture during mutation}
\end{figure}
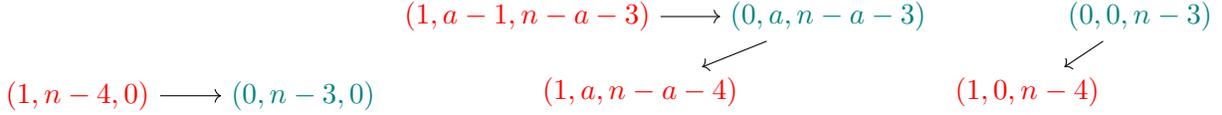

When $0<a<n-3$, then the quiver right before mutating at $(0,a,n-a-3)$ looks like the middle picture in Figure \ref{fig: local picture during mutation}. Note that by the inductive hypothesis, we have $F'_{(1,a-1,n-a-3)}=\Phi_{(a-1,0,n-a-3)}^{Q_{n-1}}$ and $F'_{(1,a,n-a-4)}=\Phi_{(a,0,n-a-4)}^{Q_{n-1}}$. For comparison, we draw these two ideal functions together in the same picture (Figure \ref{fig: F polynomial inductive step}). Now by applying the $F$-polynomial mutation formual \eqref{eq: F mutation}, we get
\[
F_{(a,0,n-a-3)}=F'_{(0,a,n-a-3)} = \Phi_{(a-1,0,n-a-3)}^{Q_{n-1}}+\Phi_{(a,0,n-a-4)}\prod_{k=0}^a X_{(a-k,k,n-a-3)}.
\]
We claim that the right hand side is the ideal function for the whole labeled poset in Figure \ref{fig: F polynomial inductive step}, which is precisely $\Phi_{(a,0,n-a-3)}$. Note that ideals in this labeled poset can be divided into two families: those that contain the vertex labeled by $X_{(0,a,n-a-3)}$ and those that do not. The ones that contain that vertex give rise to terms belong to the summand $\Phi_{(a,0,n-a-4)}\prod_{k=0}^a X_{(a-k,k,n-a-3)}$ and those that do not give rise to terms belong to the summand $\Phi_{(a-1,0,n-a-3)}^{Q_{n-1}}$. This proves the inductive step for the case $0<a<n-3$.
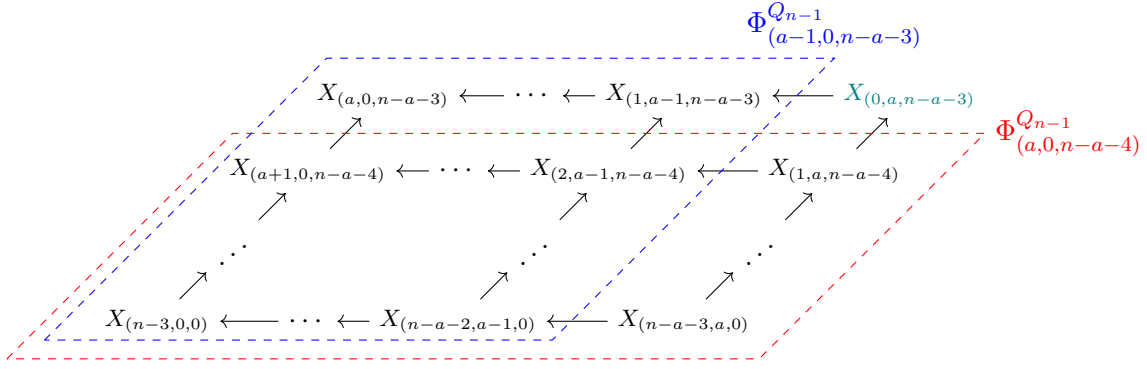
\begin{figure}[H]
    \centering
    \begin{tikzpicture}
        \node [teal] (14)  at (10,3) [] {\footnotesize{$X_{(0,a,n-a-3)}$}};
        \node (0) at (0,0) [] {\footnotesize{$X_{(n-3,0,0)}$}};
        \node (1) at (1,1) [] {$\iddots$};
        \node (2) at (2,2) [] {\footnotesize{$X_{(a+1,0,n-a-4)}$}};
        \node (3) at (4,2) [] {$\cdots$};
        \node (4) at (3,3) [] {\footnotesize{$X_{(a,0,n-a-3)}$}};
        \node (5) at (5,3) [] {$\cdots$};
        \node (6) at (7,3) [] {\footnotesize{$X_{(1,a-1,n-a-3)}$}};
        \node (7) at (6,2) [] {\footnotesize{$X_{(2,a-1,n-a-4)}$}};
        \node (8) at (9,2) [] {\footnotesize{$X_{(1,a,n-a-4)}$}}; 
        \node (9) at (5,1) [] {$\iddots$};
        \node (10) at (8,1) [] {$\iddots$};
        \node (11) at (4,0) [] {\footnotesize{$X_{(n-a-2,a-1,0)}$}};
        \node (12) at (7,0) [] {\footnotesize{$X_{(n-a-3,a,0)}$}};
        \node (13) at (2,0) [] {$\cdots$};
        \draw [->] (0) -- (1);
        \draw [->] (1) -- (2);
        \draw [->] (3) -- (2);
        \draw [->] (2) -- (4);
        \draw [->] (5) -- (4);
        \draw [->] (6) -- (5);
        \draw [->] (7) -- (3);
        \draw [->] (8) -- (7);
        \draw [->] (9) -- (7);
        \draw [->] (10) -- (8);
        \draw [->] (11) -- (9);
        \draw [->] (12) -- (10);
        \draw [->] (13) -- (0);
        \draw [->] (11) -- (13);
        \draw [->] (12) -- (11);
        \draw [->] (7) -- (6);
        \draw [->] (14) -- (6);
        \draw [->] (8) -- (14);
        \draw [dashed, red] (-2,-0.5) -- (1,2.5) -- (11,2.5) node[right] {$\Phi_{(a,0,n-a-4)}^{Q_{n-1}}$} -- (8,-0.5) -- cycle;
        \draw [dashed, blue] (-1.5,-0.25) -- (2.25,3.5) -- (9,3.5) node [above] {$\Phi_{(a-1,0,n-a-3)}^{Q_{n-1}}$} -- (5.25,-0.25) -- cycle;
    \end{tikzpicture}
    \caption{Comparison between the labeled posets for the ideal functions $\Phi_{(a-1,0,n-a-3)}^{Q_{n-1}}$ and $\Phi_{(a,0,n-a-4)}^{Q_{n-1}}$. The vertices are indexed according to $Q_n$.}
    \label{fig: F polynomial inductive step}
\end{figure}

When $a=0$, the quiver right before mutating at $(0,0,n-3)$ looks like the right picture in Figure \ref{fig: local picture during mutation}. Note that by the inductive hypothesis,
\[
F'_{(1,0,n-4)}=\Phi_{(0,0,n-4)}^{Q_{n-1}}=[X_{(1,0,n-4)} \leftarrow X_{(2,0,n-5)}\leftarrow \cdots \leftarrow X_{(n-3,0,0)}].
\]
Now by applying the $F$-polynomial mutation formula \eqref{eq: F mutation}, we get
\begin{align*}
F_{(0,0,n-3)}=&F'_{(0,0,n-3)}=1+X_{(0,0,n-3)}F'_{(1,0,n-4)}\\
=&[X_{(0,0,n-3)}\leftarrow X_{(1,0,n-4)} \leftarrow X_{(2,0,n-5)}\leftarrow \cdots \leftarrow X_{(n-3,0,0)}]=\Phi_{(0,0,n-3)}.
\end{align*}
This finishes the proof by induction.
\end{proof}

\noindent\textit{Proof of Theorem \ref{thm: F polynomials for DT of triple of flags}.} Now we are ready to prove Theorem \ref{thm: F polynomials for DT of triple of flags} for a general vertex $(a,b,c)$ in $Q_n$. Note that by Lemma \ref{lem: right most vertices}, Theorem \ref{thm: F polynomials for DT of triple of flags} is already true for the left most vertices on each row in $Q_n$. Due to the rotational symmetry of $Q_n$, we can deduce that Theorem \ref{thm: F polynomials for DT of triple of flags} is also true for all vertices that lie on the edges of $Q_n$. For the remaining quiver vertices, we will do a proof by induction by going from the bottom to top through all the rows, and going from left to right along each row. Inductively at each step for the quiver vertex $(a,b,c)$, the theorem is assumed to be true for all vertices $(x,y,z)$ with $z<c$ as well as those with $z=c$ and $y<b$. Now consider \eqref{eq3} and the second equation in \eqref{eq: DT F polynomial} for the face $f_{(a+1,b,c-1)}$. All except one factor in these equations have already been identified by the inductive hypothesis. Therefore this remaining factor, which is $\Phi_{(a,b,c)}$ in \eqref{eq3} and is $F_{(a,b,c)}$ in the second equation in \eqref{eq: DT F polynomial}, has to be identical as well. This proves the theorem.\qed

\bibliographystyle{biblio}

\bibliography{biblio}

\end{document}